\documentclass[10pt, A4, leqno, showkeys]{amsart} 
\date{\today}
\usepackage[active]{srcltx}
\usepackage[]{graphics}
\usepackage{here}
\usepackage{mathtools}
\usepackage[mathscr]{eucal}
\usepackage{mathrsfs}
\usepackage{amscd}
\usepackage{amsfonts}
\usepackage{amsmath,amsthm,amssymb}
\usepackage{latexsym}
\usepackage{comment}
\usepackage{xcolor}
\usepackage{setspace}

\numberwithin{equation}{section}
\usepackage{bm}
\numberwithin{equation}{section}
\usepackage{amsthm}
\theoremstyle{plain}
\newtheorem{theorem}{Theorem}[section]
\newtheorem{proposition}[theorem]{Proposition}

\newtheorem{corollary}[theorem]{Corollary}
\newtheorem{fact}[theorem]{Fact}

\theoremstyle{definition}
\newtheorem{definition}[theorem]{Definition}
\newtheorem{example}[theorem]{Example}
\newtheorem{remark}[theorem]{Remark}

 \newtheorem*{acknowledgements}{Acknowledgements}

\newcommand{\inner}[2]{\left\langle{#1},{#2}\right\rangle}
\newcommand{\R}{\boldsymbol{R}}

\newcommand{\D}{\boldsymbol{D}}
\newcommand{\Z}{\boldsymbol{Z}}
\newcommand{\rt}{\sqrt}
\newcommand{\fr}{\frac}
\renewcommand{\H}{\boldsymbol{H}}
\renewcommand{\Sigma}{\varSigma}

\newcommand\Res{\operatorname{Res}}

\renewcommand\Re{\operatorname{Re}}
\renewcommand\Im{\operatorname{Im}}
\newcommand\ord{\operatorname{ord}}

\newcommand{\C}{\boldsymbol{C}}
\renewcommand\tilde{\widetilde}
\renewcommand\epsilon{\varepsilon}
\newcommand{\BSigma}{\bbar{\Sigma}}

\newcommand{\bbar}{\overline}
\newcommand{\pr}{\prime}
\newcommand{\al}{\alpha}
\newcommand{\bt}{\beta}
\newcommand{\gm}{\gamma}
\newcommand{\ddd}{\cdots}
\newcommand{\zb}{\bar{z}}

\newcommand{\col}{\colon}
\newcommand{\Chat}{\widehat{\C}}

\newcommand{\tr}{\textcolor{black}}
\newcommand{\tg}{\textcolor{black}}
\newcommand{\too}{\textcolor{black}}

\newcommand{\trr}{\textcolor{black}}

\title[The classification of improper affine fronts]
{The classification of complete improper affine spheres with singularities of low total curvature and new examples}

\author{
Jun Matsumoto
}

\keywords{Improper affine sphere, Singularity, Complete, Total curvature}

\subjclass[2020]{Primary 53A15; Secondary 53A35}

\address{
Department of Mathematics, \endgraf
\too{Institute of Science Tokyo}, \endgraf
O-okayama, Meguro, Tokyo, 152-8551, Japan
}
\email{\too{matsumoto.j.d273@m.isct.ac.jp}}

\begin{document}
\maketitle

\begin{abstract}
We provide a classification of complete improper affine spheres with singularities (say \emph{improper affine fronts}) in unimodular affine three-space $\R^3$ whose total curvature is greater than or equal to $-6\pi$, and a partial classification in the case of total curvature $-8\pi$.
For the case of total curvature $-8\pi$, we give a complete classification for genus $0$ case and show the existence of an example and a one parameter family with genus $1$.
We also study the asymptotic behavior of embedded ends of complete improper affine fronts. Moreover, we give new examples for this class of surfaces, including one which satisfies the equality condition of an Osserman-type inequality and is of positive genus.
\end{abstract}

\section{Introduction}

A locally strongly convex \emph{improper affine sphere} is a surface in unimodular affine $3$-space $\R^3$ whose affine Blaschke normal vector field is parallel and affine metric is definite (see Section \ref{preliminary}). It is locally obtained as the graph of a smooth function $\varphi(x,y)$ on a planner domain satisfying the elliptic Monge--Amp\`{e}re equation
\begin{equation}\label{hess1}
\varphi_{xx}\varphi_{yy}-\varphi_{xy}^2
=1.
\end{equation}
For such surfaces, Ferrer, Mart\'{i}nez, and Mil\'{a}n established a Weierstrass-type representation formula as follows (\cite{FMM2}, \cite{FeMaMi}) (see Fact \ref{complex_representation} for more precise statement):
\begin{equation}\label{sphererep}
\psi \coloneqq \left( \bbar{F}+G,\fr{1}{2}(|G|^2 - |F|^2) + \Re\left(GF - 2\int FdG\right)\right)
\col \Sigma\to\C\times\R=\R^3,
\end{equation}
where a pair $(F,G)$ of holomorphic functions  on a Riemann surface $\Sigma$
is called \emph{Weierstrass data}.
However, as a global property, a Bernstein-type theorem for ``complete'' improper affine spheres is well known (\cite{Calabi}, \cite{Calabi2}, \cite{TW}, \cite{KN}, \cite{Kawa20}, \cite{KK_24}). That is to say, any locally strongly convex affine complete (i.e., the affine metric is definite and complete) improper affine sphere is the \emph{elliptic paraboloid} (Example \ref{paraboloid}). Thus, Mart\'{i}nez \cite{IA-map} introduced a concept of an {\it improper affine map} (Definition \ref{IA-mapdef}) (referred to as an \emph{improper affine front} in this paper), which is defined by the same representation formula \eqref{sphererep} and admits a certain kind of singularities.
He also investigated a correlation between improper affine fronts and flat fronts in hyperbolic 3-space (\cite{Galvez}, \cite{KUY1}, \cite{krsuy}) in \cite{Mart} and \cite{MM_14}. In addition, Mart\'{i}nez, Mil\'{a}n, and Tenenblat \cite{MMT} gave a new method to transform improper affine fronts by applying the theory of Ribaucour transformations.

On the other hand, Mart\'{i}nez \cite{IA-map} introduced a completeness for improper affine fronts (Definition \ref{compdef}) like other classes of surfaces with singularities (e.g., the flat fronts in hyperbolic 3-space (\cite{KUY1}), maximal surfaces in Lorentz--Minkowski 3-space (\cite{UY06_maximal}), constant mean curvature $1$ surfaces in de Sitter 3-space (\cite{Fujimori_06})) and showed that the theory of complete improper affine fronts shares numerous global properties with the theory of complete minimal surfaces in Euclidian 3-space. As one of them, he proved a Huber--Osserman-type theorem (\cite{Huber_57}, \cite{Os}). Namely, the Riemann surface $\Sigma$ is biholomorphic to a compact Riemann surface $\BSigma$ minus finite points\trr{,} and the Weierstrass data
$(F,G)$ can be extended meromorphically to $\BSigma$ (Fact \ref{Huber2}). Moreover, he proposed a total curvature for complete improper affine fronts with respect to a certain complete Riemannian metric expressed in terms of \tr{their} Weierstrass data and showed an Osserman-type inequality
\begin{equation}
-\fr{1}{2\pi}\ \text{TC}(\Sigma)\geq -\chi(\BSigma)+2(\text{number of ends}),
\end{equation}
where TC($\Sigma$) denotes the total curvature and $\chi(\BSigma)$ is the Euler \tr{characteristic} of $\BSigma$
(Fact \ref{property}). 
In this paper, we study the following two topics for complete \trr{improper} affine fronts related to \tr{the results for minimal surfaces}.

Firstly, we describe an \trr{asymptotic} behavior of 
\tg{embedded ends} of \tg{complete} improper affine fronts in Section \ref{asymptoticsection}.
Schoen \cite{Schoen} proved that 
\tg{embedded ends} of \tg{complete} minimal surfaces in Euclidian 3-space with finite total curvature is asymptotic to either the plane or the catenoid.
\trr{Also}, Jorge and Meeks \cite{jorge} showed a relation between \trr{the} embeddedness of ends and the equality of the Osserman inequality and constructed the surface with high symmetry that attains the equality condition of the inequality.
As affine correspondence to these results, we define a concept \tr{of ``asymptoticity''} for embedded ends of complete improper affine fronts and classify asymptotic \trr{classes} of embedded ends into three types (Theorem \ref{asymptotic}). It is associated with an equality condition of the Osserman-type inequality (Corollary \ref{equality-condition}).
\trr{Moreover,} we \trr{construct} new examples  (Examples \ref{46noid}, and \ref{4noidiikanji}) with embedded ends, which satisfy the equality condition of the Osserman-type inequality.

Secondly, we study a classification of complete improper affine fronts in terms of the total curvature in Section \ref{classificationsection}. Complete orientable minimal surfaces in Euclidian 3-space of low total curvature were classified by Osserman \cite{Os} and L\'{o}pez \cite{Lopez}.  
So, this leads a natural problem ``to classify complete improper affine fronts with low total curvature''. The total curvature of improper affine fronts is $-2m\pi$, where $m$ is the mapping degree of a certain holomorphic map called the \emph{Lagrangian Gauss map} (Section \ref{preliminary}). Here, we classify the surfaces of the total curvature greater than or equal to $-8\pi$. Namely, the main result of this paper is the following:
\begin{theorem}\label{main}
\begin{itemize}
\item
Complete improper affine fronts in $\R^3$ whose total curvature is greater than or equal to $-6\pi$ are all genus $0$ and constructed \tr{by} the Weierstrass data as in Theorems \ref{tc0}, \ref{tc2}, \ref{tc4}, \ref{tc6}.
\item
Genus of complete improper affine fronts with the total curvature $-8\pi$ is less than or equal to $1$.

\item
Complete improper affine fronts in $\R^3$ with the total curvature $-8\pi$ and \underline{genus $0$} are the surfaces described in Theorem \ref{tc80}.
\item
There exists a complete improper affine front in $\R^3$ with the total curvature $-8\pi$ and \underline{genus $1$}.
In particular, it is the one with the maximum total curvature and positive genus (Proposition \ref{torus8pi1} and Theorem \ref{torus8pi2}).
\end{itemize}
\end{theorem}
We show the existence in a special case for the fourth statement of Theorem \ref{main}. 
In addition, we \tr{show} that there is at least one parameter family of complete improper affine fronts with the total curvature $-8\pi$ and genus $1$, each of which has different \trr{complex} structure (Remark \ref{notunique}).
The complete classification is an open problem in the genus $1$ case.

\trr{Furthermore}, the only known example of complete improper affine fronts of genus $1$ was composed by Mart\'{i}nez in \cite[Section 4, No.6]{IA-map},  whose total curvature is $-12\pi$. \tr{At the end} of Section \ref{classificationsection}, we give a new example of genus $1$ surface whose total curvature is $-10\pi$.

\begin{acknowledgements}
The author would like to express his gratitude to Kotaro Yamada for his helpful advice and comments on the author's research.
In addition,  the author wishes to thank Masaaki Umehara, Yu Kawakami, and Shunsuke Kasao for their valuable comments and fruitful discussions in developing this work.
\tr{The author also sincerely thanks the referees for their insightful feedback and suggestions to improve this paper.}

This work was supported by JST SPRING, Japan Grant Number JPMJSP2180.
\end{acknowledgements}

\section{Preliminaries}\label{preliminary}
Firstly, we will briefly describe some definitions and fundamental facts about \trr{the} geometry of affine immersions in unimodular affine 3-space $\R^3$ (see \cite{simon} and \cite{nomizu} for \trr{details}).
Let $\Sigma$ be a connected and oriented
2-manifold, $\psi \colon \Sigma \to \R^3$ an immersion, and $\xi$ a vector field of $\R^3$ along $\psi$ which is transversal to $d\psi(T\Sigma)$. 
Then, there uniquely exist a torsion-free affine connection $\nabla$, a symmetric $(0,2)$-tensor $h$, a $(1,1)$-tensor $S$, and a $1$-form
$\tau$ on $\Sigma$, which satisfy
\begin{equation}
\left\{
\begin{array}{l}
D_Xd\psi(Y) = d\psi(\nabla_XY) + h(X,Y) \xi,\\
D_X\xi = -d\psi(S(X)) + \tau(X) \xi,
\end{array}
\right.
\end{equation}
where $D$ is the canonical connection of $\R^3,$ and $X,Y$ are vector fields on $\Sigma$. Here, $h$ is called the \emph{affine metric} of $\psi$ with respect to $\xi$.
When $h$ is definite, $\psi$ is said to be \emph{locally strongly convex}.
\tr{From now on}, we will only consider the case that $\psi$ is locally strongly convex (for the indefinite case, see \cite{nakajo}, \cite{Milan3}, \cite{Milan2}, \cite{MM1}, \cite{Mi_20_inde_affine_Cauchy}). For given a locally strongly convex immersion $\psi$, one can uniquely choose the transversal vector field $\xi$ which satisfies 
\begin{equation}
\left\{
\begin{array}{l}\label{xicondition}
D_X\xi = -d\psi(S(X)),\\
\det(d\psi(X), d\psi(Y), \xi) =(h(X,X)h(Y,Y)-h(X,Y)^2)^{1/2},
\end{array}
\right.
\end{equation}
where $\det$ denotes the determinant function of $\R^3$.
The transversal vector field $\xi$ which satisfies \eqref{xicondition} is called the
\emph{affine normal vector field}, and the pair $(\psi, \xi)$ (or simply $\psi$)  is called the \emph{Blaschke immersion}. 
A Blaschke immersion $\psi$ is said to be an \emph{improper affine sphere} if $S=0$ holds in \eqref{xicondition}. 
Then, after equiaffine transformations of $\R^3$ $(\R^3\ni\bm{x}\mapsto A\bm{x}+\bm{b}\in\R^3$, where $A\in \tr{\text{SL}(}\tg{3}\tr{, \R)}$ 
and $\bm{b}\in\R^3$), we can take the affine normal vector field $\xi$ as $\xi=(0,0,1)$.

Next, for any improper affine sphere $\psi\col\Sigma\to\R^3$, considering the conformal structure induced by the affine metric $h$, we regard $\Sigma$ as a Riemann surface. In \cite{IA-map}, Mart\'{i}nez introduced the following complex representation formula for improper affine spheres similar to the Weierstrass representation formula for minimal surfaces in Euclidian 3-space (see \cite{osserman}):

\begin{fact}{\cite[Theorem 4]{FMM2}, \cite[Lemma 1]{FeMaMi}, \cite[Theorem 3]{IA-map}}\label{complex_representation}
Let $\Sigma$ be a Riemann surface, and $(F,G):\Sigma\to\C^2$  a complex regular curve (that is, $F$ and $G$ are holomorphic functions satisfying $(dF, dG) \neq (0, 0)$) which satisfies $|dF|\neq|dG|$ and
$
\Re \int_{\gm} FdG = 0
$
for any closed curve $\gm$ in $\Sigma$. 
Then,
\begin{equation}\label{IAsphererep}
\psi \coloneqq \left( \bbar{F}+G,\fr{1}{2}(|G|^2 - |F|^2) + \Re\left(GF - 2\int FdG\right)\right)
: \Sigma \to \C \times \R=\R^3
\end{equation}
gives an improper affine sphere with the affine normal vector field $\xi=(0,0,1)$. Conversely, any improper affine spheres $\psi: \Sigma\to\R^3$ with the affine normal $\xi=(0,0,1)$ are given in this way, and the complex structure of the 2-manifold $\Sigma$ is compatible with $h$.
\end{fact}

The pair of holomorphic functions $(F, G)$ is called the \emph{Weierstrass data} of $\psi$.
We find that the metric $ds^2$ represented as
\begin{equation}\label{flatff}
ds^2\coloneqq \inner{d\mathcal{X}}{d\mathcal{X}}=|dF|^2+|dG|^2+dFdG+\bbar{dFdG}
\end{equation}
is a non-degenerate flat metric, where $\mathcal{X}\coloneqq \bbar{F}+G$ is the \tr{first two  components} in \eqref{IAsphererep}, and $\inner{\ }{\ }$ is the standard Euclidian inner product of $\C=\R^2$ under \tr{the} canonical identification. This metric $ds^2$ is called the \emph{flat fundamental form}.
Also, the affine metric $h$ can be expressed as $h=|dG|^2-|dF|^2$.
When $|dG|=|dF|$ holds at a point (i.e., the affine metric $h$ degenerates), the improper affine sphere $\psi$ represented by \eqref{IAsphererep} is not immersed. \trr{The} point also corresponds to the point where the flat fundamental form $ds^2$ degenerates. Hence, using the notations above, Mart\'{i}nez introduced the following concept of improper affine maps, which is a generalization of improper affine spheres in the sense of admitting singularities.

\begin{definition}{\cite[Definition 1]{IA-map}}\label{IA-mapdef}
Let $\Sigma$ be a Riemann surface and $(F,G) : \Sigma \to \C^2$ a complex regular curve satisfying the \emph{period condition}
\begin{equation}\label{period}
\Re\int_{\gm} FdG = 0
\end{equation}
for any closed curve $\gm$ in $\Sigma$.
Then, the map $\psi\colon \Sigma \to \C \times \R=\R^3$ given by
\begin{equation}\label{wrep}
\psi \coloneqq \left(G + \bbar{F},\fr{1}{2}(|G|^2 - |F|^2) + \Re\left(GF - 2\int FdG\right)\right)
\end{equation}
is called an \emph{improper affine map}.
\end{definition}
The singular points of an improper affine map correspond with the points where the affine metric $h$ degenerates and the points where $ds^2$ degenerates.
As shown in \cite{nakajo} and \cite{completeness_lemma}, an improper affine map \tr{is} a (wave) front.
Thus, in this \tr{sence}, we call the improper affine map the \emph{improper affine front} in this paper,
which \trr{was} first referred to as such in \cite{KN}.
The differential geometry of fronts is discussed in \cite{SUY_09_front} and \cite{yamadasing}.
We \tr{note} that the improper affine front \tr{is} \tg{in} a special \tr{class} of affine maximal surfaces with singularities \tr{which are called affine maximal maps}, \tg{defined and investigated in
\cite{AMM_09_affine_Cauchy}, \cite{AMM_09_affine_maximal_map}, \cite{AMM_09_non_removable},
and \cite{AMM_11_extension_Bernstein}.}

\begin{remark}\label{equiaffineremark}
For given an improper affine front $\psi\col\Sigma\to\R^3$ with Weierstrass data $(F,G)$, another improper affine front constructed from $(\widetilde{F},\widetilde{G})$ defined by
\begin{equation}\label{equiaffine}
(\widetilde{F},\widetilde{G})
\coloneqq(\al F+\bt G+\mu, \bbar{\bt}F+\bbar{\al}G+\lambda)
\quad
(\al, \bt, \mu, \lambda \in \C, |\al|^2 - |\bt|^2 = 1)
\end{equation}
gives an equiaffinely equivalent improper affine front.
In particular, for any $\mu,\lambda\in\C$, $(F+\mu, G+\lambda)$ gives \tr{a parallel translation} \tg{of $\psi$} \tr{in $\R^3$.} 
Conversely, any improper affine fronts which \tr{transform} to $\psi$ by an equiaffine transformation
whose differential map preserves the affine normal vector $\xi=(0,0,1)$ are given in this way (\cite{Ferrer1}).
\end{remark}

From now on, $\psi\col\Sigma\to\C\times\R=\R^3$ is an improper affine front with Weierstrass data $(F, G)$.
Next, we shall review the concepts of completeness and some properties for complete improper affine fronts, shown in \cite{IA-map}, which play important roles in this paper.
\begin{definition}{\cite[Definition 2]{IA-map}, \cite[Definition 3.1]{KUY1}}\label{compdef}
An improper affine front $\psi\colon \Sigma \to \C \times \R$ is said to be \emph{complete}
if there exists a symmetric bilinear form $T$ with a compact support such that
\begin{equation}\label{completemetric}
\widetilde{ds}^2 \coloneqq T + ds^2
\end{equation}\label{complete metric}
is a complete Riemannian metric on $\Sigma$, 
where $ds^2$ is the flat fundamental form.
\end{definition}
\begin{fact}{\cite[Proposition 1]{IA-map}}\label{Huber2}
Let $\psi\colon \Sigma \to \C \times \R$ be a complete improper affine front. 
Then,
$\Sigma$ is biholomorphic to $\BSigma \setminus \{p_1, \tr{\dots}, p_n\}$,
where $\BSigma$ is a compact Riemann surface, and $n\geq1$ is an integer.
Moreover, the Weierstrass data $(F, G)$ of $\psi$ can be extended meromorphically to $\BSigma$.
In particular, $F$ and $G$ have at most a pole at each $p_j$.
\end{fact}

Each puncture point $p_j$ is called an \emph{end} \tr{of an improper affine front}.
\too{In this paper, we study the end such that $ds^2$ is non-degenerate on $U\setminus\{p_j\}$ (i.e., $\psi$ is an improper affine sphere there) when $U$ is a sufficiently small neighborhood of $p_j$ (Section \ref{asymptoticsection}).}
\tg{Moreover}, an end $p$ of $\psi$ is said to be an \emph{embedded end} if there is a small neighborhood $U$ of $p$ such that $\psi |_{U\setminus \{p\}}$ is an embedding.

\tr{We} set $\Sigma = \BSigma_g \setminus \{p_1, \tr{\dots}, p_n\}$, where $\BSigma_g$ is a compact Riemann surface of genus $g \ \tg{(\geq 0)}$, and
let $\psi\colon \Sigma \to \C \times \R = \R^3$ be a complete improper affine front. Here,
\begin{equation}\label{gaussmap}
\rho\coloneqq\fr{dF}{dG}
\end{equation}
defines a meromorphic function on $\Sigma$, and $\rho$ is termed the \emph{Lagrangian Gauss map}.
When we set
$$
\mathcal{L}\coloneqq \mathcal{X}+i\mathcal{N},
$$
where
$\mathcal{X}=\bbar{F}+G$ and $\mathcal{N}\coloneqq\bbar{F}-G$, the map $\mathcal{L}\col\Sigma\to\C^2$ \tr{defines} the special Lagrangian immersion.
The induced metric $d\tau^2$ from $\C^2$ given by
\begin{equation}\label{dtau2}
d\tau^2\coloneqq\mathcal{L}^\ast\inner{}{}_{\C^2}=\inner{d\mathcal{X}}{d\mathcal{X}}+\inner{d\mathcal{N}}{d\mathcal{N}}=2(|dF|^2+|dG|^2)
\end{equation}
is a complete Riemannian metric and conformal to $h$ at points where $h$ is non-degenerate, where
$\inner{}{}_{\C^2}$ stands for the standard Euclidian inner product of $\C^2=\R^4$ (\cite{IA-map}, Theorem 1).
\tg{Mart\'{i}nez and Mil\'{a}n} \too{pointed out} \tg{a local correspondence between improper affine fronts and flat fronts in hyperbolic 3-space} \too{in} \tg{\cite{Mart} and \cite{MM_14}.
That is, \trr{if} we set $\omega = dF, \ \theta= dG$, then $(\omega, \theta)$ locally gives a Weierstrass data of a flat front.
Then, the metric $d\tau^2 = 2(|\omega|^2+|\theta|^2)$ coincides with the pull-back of the Sasakian metric on the unit cotangent bundle of hyperbolic 3-space by the Legendrian lift of the flat front (see \cite{KUY1}).}

\begin{fact}{\cite[Section 3]{IA-map}}\label{property}
A \trr{complete} improper affine front $\psi\colon \Sigma=\BSigma_g \setminus \{p_1, \tr{\dots}, p_n\} \to \C \times \R$ satisfies the following properties:
\begin{itemize}
\item
An end $p_j\ (j=1,2,\tr{\dots}, n)$ is embedded if and only if $F$ and $G$ have at most a simple pole at $p_j$.
\item \emph{(Osserman-type inequality)}
When we denote by $K_{\tau}$ and $dA_{\tau}$ the Gaussian curvature and the area element with respect to $d\tau^2$, it holds that
\begin{equation}\label{ossermanineq1}
-\fr{1}{2\pi} \int_\Sigma K_{\tau} dA_{\tau} \geq -\chi(\BSigma_g) + 2n,
\end{equation}
where $\chi(\BSigma_g)=2-2g$ is the Euler \tr{characteristic} of $\BSigma_g$.
\tr{Moreover, the equality in \eqref{ossermanineq1} holds if and only if all ends are embedded.}
\end{itemize}
\end{fact}

The integral of $K_\tau dA_\tau$ in the left-hand side of \eqref{ossermanineq1} is called the \emph{total curvature} of $\psi$. \cite[Theorem 1.1, \trr{Corollary} 1.2]{Kawa_13} shows that the Gaussian curvature $K_{\tau}$ with respect to the conformal metric
$d\tau^2=2(|dF|^2+|dG|^2)=2(1+|\rho|^2)|dG|^2$ is
\begin{equation}\label{Gaussian curvature}
K_\tau = - \fr{1}{(1+|\rho|^2)^3} \left|\fr{d\rho}{dG}\right|^2.
\end{equation}
Hence, one can \trr{verify} that the total curvature satisfies
\begin{equation}\label{tc}
\int_\Sigma K_{\tau} dA_{\tau} = -2\pi\deg\rho \in-2\pi\Z_{\geq 0},
\end{equation}%
\too{where $\deg\rho$ denotes the degree of the meromorphic function $\rho$, } and the Osserman-type inequality \eqref{ossermanineq1} can be rewritten as
\begin{equation}\label{ossermanineq2}
\deg\rho \geq 2(g-1+n).
\end{equation}%
\tg{
\begin{remark}
By the proof of \cite[Theorem 4]{IA-map}, the total curvature satisfies the following Jorge--Meeks-type formula (\cite{jorge}):
\begin{equation}\label{JMformula}
-\fr{1}{2\pi}\int_{\Sigma}K_\tau d_A =\deg \rho = -\chi\left(\BSigma_g\right)+\sum_{j=1}^n \max\{m_j+1, n_j+1\},
\end{equation}
where the positive integers $m_j$ and $n_j$ stand for the pole order at the end $p_j$ of $F$ and $G$, respectively.
The proof of \cite[Proposition 2]{IA-map} shows that the map $\mathcal{X} : \Sigma \to \C=\R^2$ \trr{winds} $\max\{m_j, n_j\}$-times along any closed curve around $p_j$. %
\too{
In particular, when $\max\{m_j, n_j\} = 1$ holds for any $j\in\{1,2, \dots, n\}$ (i.e., all ends are embedded), \eqref{JMformula} gives the equality of the Osserman-type inequality.
}
\end{remark}%
}%
Other global results for complete improper affine fronts are also investigated in \cite{MM_14_geometric_aspects} and \cite{ACG_07_Cauchy}.

\tr{At} the end of this section, we \tr{recall} some examples of complete improper affine fronts with only embedded ends.

\begin{example}{\cite[Section 4]{IA-map}}\label{paraboloid}
A complete improper affine front \tr{with} the Weierstrass data $(F,G)$ given by
\begin{equation}\label{PARA}
F = \fr{1}{z},
\quad
G = \fr{b}{z}
\quad 
(b \in \C, \ |b| \neq 1)
\end{equation}
\tr{on $\Sigma = \Chat \setminus \{0\} \ (\Chat\coloneqq\C\cup\{\infty\})$} is called an \emph{elliptic paraboloid} (Figure \Ref{martex}, (a)). Moreover, for each $b$, we find that the elliptic paraboloid is equiaffinely equivalent to a rotational paraboloid\tg{, i.e., the case of $b=0$ in \eqref{PARA}}.
\end{example}

\begin{example}{\cite[Section 4]{IA-map}}\label{R-IA}
A complete improper affine front \tr{with} the Weierstrass data $(F,G)$ given by
\begin{equation}\label{RIA}
F = \fr{1}{z},
\quad
G = a z
\quad
( a \in \R \setminus \{0\})
\end{equation}
\tr{on $\Sigma = \C \setminus \{0\}$} is called a \emph{rotational improper affine front}
(Figure \Ref{martex}, (b)).
\end{example}

\begin{example}{\cite[Section 4]{IA-map}}\label{NR-IA}
A complete improper affine front \tr{with} the Weierstrass data $(F,G)$ given by
\begin{equation}\label{NRIA}
F = \fr{1}{z},
\quad
G = a z + \fr{b}{z}
\quad
(a \in \R \setminus \{0\}, \ b \in \C \setminus \{0\},\   |b| \neq 1)
\end{equation}
\tr{on $\Sigma = \C \setminus \{0\}$} is called a \emph{non-rotational improper affine front}
(Figure \Ref{martex}, (c)).
\end{example}

\begin{center}
\begin{figure}[h]
\begin{tabular}{c@{\hspace{1.5cm}}c@{\hspace{1.5cm}}c@{\hspace{1.5cm}}c}
\includegraphics[width=27mm]{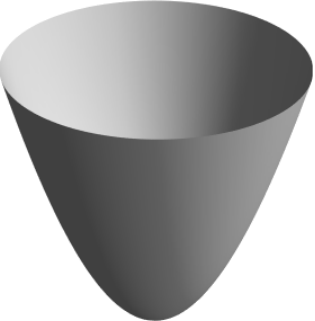}&
\includegraphics[width=24mm]{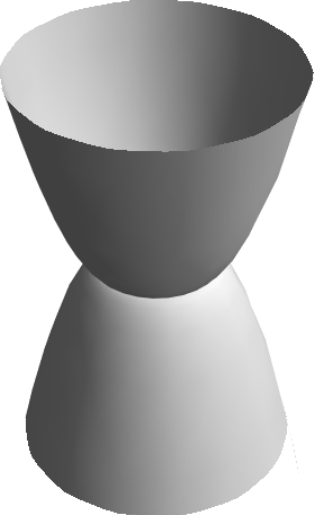}&
\includegraphics[width=27mm]{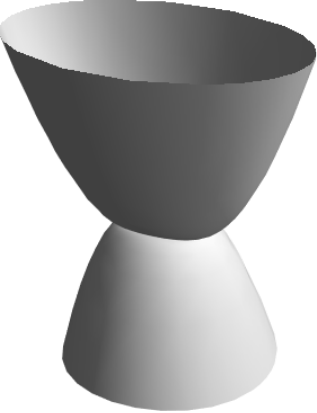}\\
(a) Example \Ref{paraboloid} & (b) Example \Ref{R-IA} & (c) Example \Ref{NR-IA} 
\end{tabular}
\caption{Complete improper affine fronts with only embedded ends}\label{martex}
\end{figure}
\end{center}

\section{Asymptotic behavior of complete embedded ends}\label{asymptoticsection}

\subsection{Asymptotic behavior of the complete embedded ends}

As mentioned in Fact \ref{property}, the equality condition of the Osserman-type inequality \eqref{ossermanineq1} is equivalent to the condition that all ends are embedded. 
Using this fact and the condition for an end to be embedded, we will classify complete embedded ends in the sense of \tg{asymptoticity} into three types \trr{similarly} to \cite{yamadaineq}.
Throughout this section, we take a sufficiently small local complex coordinate neighborhood centered at an end of a complete improper affine front and consider a local expression of an improper affine front
$\psi\colon  \D_{\varepsilon}^{\ast}\coloneqq\{\tg{z\in\C\ ;\ }0<|z|<\epsilon\}  \to \C \times \R$
($\varepsilon>0$) with Weierstrass data $(F,G)$. 
By completeness, we may assume that $\D_{\varepsilon}^\ast$ does not contain the singular set of $\psi$, that is, the flat fundamental form $ds^2$ is non-degenerate on $\D_{\varepsilon}^\ast$. 
Then, $\psi$ is complete at $0$ (that is, the length with respect to $ds^2$ of any path in $\D_{\varepsilon}^{\ast}$ which accumulates to $0$ diverges to $\infty$). 
Hence, we may also assume that $\psi$ is an improper affine sphere on $\D_{\varepsilon}^{\ast}$.

\begin{definition}
An improper affine sphere $\psi\colon \D_{\varepsilon}^{\ast}\to \C \times \R$ \tg{which is} complete at $0$ is said to be \emph{asymptotic} to a \emph{type-P end} (resp. \emph{type-R end}, \emph{type-NR end}) if there exists a piece of an elliptic paraboloid \eqref{PARA} (resp. a rotational improper affine front \eqref{RIA}, a non-rotational improper affine front \eqref{NRIA})
$$
\widetilde{\psi}\colon \D_{\varepsilon}^{\ast} \to \C \times \R
$$
which is complete at $0$ such that
\begin{equation}
|\psi(z) - \widetilde{\psi}(z)| = \trr{o(1)}
\end{equation}
holds, where $o(1)$ means \tr{the} term\tr{s} tending to $0$ as $z \rightarrow 0$, and $|\cdot|$ is the standard Euclidian norm.
Here, we regard $\C\times\R=\R^3$ as Euclidian space, not affine space.
\end{definition}

\begin{theorem}\label{asymptotic}
Let $\psi\colon \D_{\varepsilon}^{\ast}\to\R^3$ be an improper affine sphere \trr{which is} complete at $0$. Then, the end $0$ is embedded end
if and only if $\psi$ is asymptotic to one of the type-P end, type-R end, or type-NR end.
\end{theorem}

\begin{proof}
Assume that the end $0$ is embedded.
By Fact \ref{property}, $F$ and $G$ can be expanded around $z=0$ as
\begin{equation}
F = \fr{a_{-1}}{z} + \sum^{\infty}_{n=0} a_n z^n,
\qquad
G= \fr{b_{-1}}{z} + \sum^{\infty}_{n=0} b_n z^n,
\end{equation}
and by exchanging $F$ and $G$ if necessary, we may assume that $a_{-1}\neq 0$.
After a parallel translation of $\R^3$
and coordinate changes on $\D_{\epsilon}^{\ast}$,
we \tr{may} suppose that  $F = 1/z$ and $b_0=0$.
In addition, the period condition \eqref{period} is equivalent to $b_1\in\R$.

By \eqref{wrep}, we obtain the concrete expression of $\psi$ as,
\tg{
\begin{multline*}
\psi(z)=
\left(b_1z+\fr{b_{-1}}{z}
\fr{1}{\bar{z}}, \right. \\
\left.
\fr{1}{2}
\left(
b_1^2|z|^2-\fr{1}{|z|^2}
\left(\trr{1-}|b_{-1}|^2-2b_1\Re(\overline{b_{-1}}z^2)\right)
\right)
-2b_1\log |z|
\right)
+o(1)
\end{multline*}
}%
up to additive constant vectors of $\R^3$.
We divide the situation into the following two cases.

\underline{{\bf Case 1}}
\quad
The case of $b_{-1} = 0$.

{\bf (I)}
\
When $b_1 \neq 0$,
$\psi$ is asymptotic to the type-R end.
In fact, \tr{by} Example \Ref{R-IA},
the rotational improper affine front with the Weierstrass data
$F = 1/z, \  G=b_1z \ (b_1 \in \R\setminus\{0\})$ is expressed as
\begin{equation}\label{psiR}
\tilde{\psi}_{\text{R}} (z)\coloneqq
\left(
b_1z+\fr{1}{\bar{z}},\ \fr{1}{2}\left(b_1^2|z|^2-\fr{1}{|z|^2}\right)-2b_1\log |z|
\right)
\end{equation}
up to an additive constant vector, and it holds that
$$
|\psi(z)-\tilde{\psi}_{\text{R}}(z)| = \trr{o(1)}.
$$

{\bf (II)}
\
When $b_1 = 0$,
$\psi$ is asymptotic to the type-P end. In fact, from Example \Ref{paraboloid}, Weierstrass  data of the elliptic paraboloid for $b=0$, $F=1/z, \  G=0$
corresponds to the surface
\begin{equation}
\tilde{\psi}_{\check{\text{P}}}(z) \coloneqq
\left(
\fr{1}{\bar{z}},\ -\fr{1}{2|z|^2}
\right),
\end{equation}
from \eqref{wrep}, and one can obtain
$$
|\psi(z)-\tilde{\psi}_{\check{\text{P}}}(z)|=o(1).
$$

\underline{{\bf Case 2}}
\quad
The case of $b_{-1} \neq 0.$

{\bf (I)}
\
When $b_1 \neq 0$,
$\psi$ is asymptotic to the type-NR end. In fact, from Example \Ref{NR-IA}, Weierstrass  data is given by
$F=1/z, \  G=b_1z+b_{-1}/z \  (b_1\in\R\setminus\{0\}, b_{-1}\in\C\setminus\{0\}, |b_{-1}|\neq1).$
Then, this surface is expressed as
\tg{
\begin{multline}\label{psiNR}
\tilde{\psi}_{\text{NR}}(z)
\coloneqq
\left(b_1z+\fr{b_{-1}}{z}+
\fr{1}{\bar{z}}, \right. \\
\left.
\fr{1}{2}
\left(
b_1^2|z|^2-\fr{1}{|z|^2}
\left(\trr{1-}|b_{-1}|^2-2b_1\Re(\overline{b_{-1}}z^2)\right)
\right)
-2b_1\log |z|
\right)
\end{multline}
}%
up to an additive constant vector.
Hence, we have
$$
|\psi(z)-\tilde{\psi}_{\text{NR}}(z)|=o(1).\\
$$

{\bf (II)}
\
When $b_1 = 0$, $\psi$ is asymptotic to the type-P end. Indeed, from Example \Ref{paraboloid}, Weierstrass  data of elliptic paraboloid for $b=b_{-1}$ is given by
$F=1/z, \ G=b_{-1}/z \ (|b_{-1}|\neq1),$
and then this surface can be expressed as
\begin{equation}\label{psiP}
\tilde{\psi}_{\text{P}}(z)\coloneqq
\left(\fr{b_{-1}}{z}+\fr{1}{\zb},\ \fr{1}{2|z|^2}(|b_{-1}|^2-1)\right)
\end{equation}
from \eqref{wrep}. 
Therefore, we find
$$
|\psi(z)-\tilde{\psi}_{\text{P}}(z)|=o(1).
$$

Conversely,
we suppose that an improper affine sphere $\psi\col\D_{\epsilon}^\ast\to\R^3$ \tg{which is} complete at $0$ is asymptotic to one of those three types.
Now, assume that $0$ is not an embedded end.
Then, from Fact \Ref{property}, $F$ and $G$ can be expanded to
$$
F=\sum_{n=\tr{-k}}^{\infty}a_nz^n,
\quad
G=\sum_{n=\tr{-l}}^{\infty}b_nz^n
\quad (a_{-k},b_{-l} \neq 0, \ \tr{k\geq l,\ k\geq 2})
$$
around $z=0$.
\tr{Similarly}, 
we may assume that $F = 1/z^k, b_0=0$,
and the period condition is equivalent to $b_k \in \R$.
Putting $\psi =(\psi_1+i\psi_2, \psi_3) \in \C \times \R$, we can compute $\psi_3(z)$ as
$$
\psi_3(z)
=
\fr{1}{|z|^{2k}}\left(-\fr{1}{2}+O(1)\right)
$$
\tr{by} \eqref{wrep}, where $O(1)$ is \tr{the} bounded term\tr{s} as $z\to0$. If $\psi$ is asymptotic to the type-R end, then from \eqref{psiR}, we have
$$
|\psi_3(z)-(\psi_{\text{R}})_3(z)|
=\fr{1}{|z|^{2k}}\left|-\fr{1}{2}+O(1)\right|
\to
\infty
\quad
(z \to 0).
$$

This contradicts the assumption of asymptoticity. Similarly, we can lead contradictions \trr{in} the cases of
the type-P end and the type-NR end. Therefore, we obtain the conclusion.
\end{proof}

Combining the equality condition of the Osserman-type inequality \eqref{ossermanineq1} with Theorem \ref{asymptotic}, one can directly show the following corollary:

\begin{corollary}\label{equality-condition}
A complete improper affine front in $\C \times \R = \R^3$ attains the equality in the Osserman-type inequality \eqref{ossermanineq1} if and only if each end is asymptotic to one of the type-P end, type-R end, or type-NR end.
\end{corollary}

Symmetry, uniqueness of solutions of the exterior Plateau \trr{problem} associated to \eqref{hess1}, and maximum principle at infinity for improper affine spheres are studied in \cite{FMM2} and \cite{FeMaMi}.

\subsection{New examples with embedded ends}

We introduce new examples of complete improper affine fronts with embedded ends as a correspondence to the Jorge--Meeks minimal surface in Euclidian 3-space (\cite{jorge}). 

Let
$n\geq2$ be an integer and $\Sigma \coloneqq \widehat{\C} \setminus\{1, \zeta,\tr{\dots},\zeta^{n-1}, \eta, \eta \zeta,\tr{\dots},\eta \zeta^{n-1}\}$, where
$\zeta \coloneqq \exp(2\pi i/n), \eta \coloneqq \exp(\pi i/n)$ and $\Chat=\C\cup\{\infty\},$ and consider $(F,G)$ given by
\begin{equation}\label{2n_embedded}
F = \sum_{j=0}^{n-1} \fr{\al_j}{z-\zeta^j},
\qquad
G=\sum_{\tg{k}=0}^{n-1} \fr{\bt_{\tg{k}}}{z-\eta \zeta^{\tg{k}}},
\end{equation}
where $\al_j, \bt_{\tg{k}} \in \C \setminus \{0\}$, and $z$ is the canonical complex coordinate of $\C$. 
We obtain the following examples by \tg{an appropriate choice of}  \tr{complex numbers} $\al_j$ and $\bt_k$.

\begin{example}\label{46noid}
We choose $\al_j, \bt_k$ satisfying
$$
\al_j \eta^{n-1} \zeta^{n-j},
\quad
\bt_k \zeta^{n-k}
\in \R
\qquad
(j, k = 0,...,n-1).
$$
For example, we \tr{set}
$
\al_j =\lambda_j \eta \zeta^j,
\bt_k = \mu_k\zeta^k,
$
where $\lambda_j, \mu_k\in\R\setminus\{0\}$.
Then, $(F,G)$ in (\Ref{2n_embedded}) \tr{gives} a complete improper affine front $\psi\colon\Sigma\to\R^3$ with $2n$ embedded ends
(Figure \Ref{2nnoid}: (a) $n=2;\  \tg{\lambda_0=\mu_0=1,\ \lambda_1=\mu_1=-1}$
(b) $n=3;\  \tg{\lambda_0=\lambda_1=\mu_1=\mu_2}$\\
$\tg{=1/5,\ \lambda_2=\mu_0=-1}$).
\end{example}

Moreover, we consider the following example in the special case of $n=2$. Set $\al_j=\bt_{\tg{k}}=1\ (j, \tg{k}=0,1)$ in \eqref{2n_embedded}. 
\begin{example}\label{4noidiikanji}
Let
\begin{equation}\label{4noiddata}
F=\fr{1}{z-1}+\fr{1}{z+1}, \qquad G=\fr{1}{z-i}+\fr{1}{z+i}
\end{equation}
\tg{on $\Sigma=\widehat{\C}\setminus\{\pm1, \pm i\}$.}
Since
the residues of $FdG$ at each ends are
$\Res(FdG,\pm1)=\Res(FdG,\pm i)=0$, the period condition \eqref{period} is satisfied.
Hence, \eqref{4noiddata} \tr{gives} a complete improper affine front with four embedded ends.
This surface is not equiaffinely equivalent to the surface in Example \ref{46noid}, $n=2$ \tg{because of Remark \ref{equiaffineremark}}.
(Figure \Ref{2nnoid}, (c))
\end{example}

\begin{figure}[h]
\begin{center}
\begin{tabular}{c@{\hspace{1cm}}c@{\hspace{1cm}}c}
\includegraphics[width=25mm]{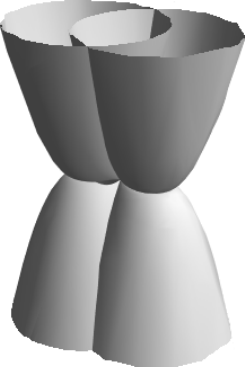}&
\includegraphics[width=23mm]{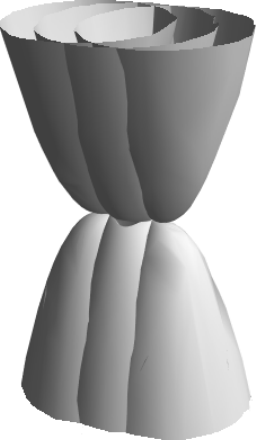}&
\includegraphics[width=23mm]{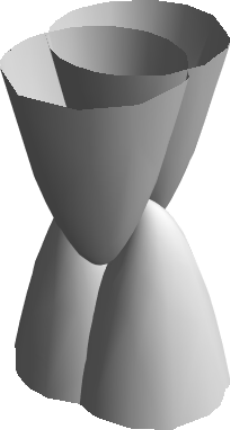}\\
(a) Example \ref{46noid}, $n=2$ & (b) Example \ref{46noid}, $n=3$ & (c) Example \ref{4noidiikanji}
\end{tabular}
\end{center}
\caption{New examples with $2n$ embedded ends}\label{2nnoid}
\end{figure}

\section{Classification of complete improper affine fronts of total curvature $-2m\pi$}\label{classificationsection}

In the minimal surface theory, Osserman \cite{Os} and L\'{o}pez \cite{Lopez} classified complete minimal surfaces in the Euclidian 3-space whose total curvature is $-4m\pi\ (0\leq m\leq2)$ (that is, the mapping degree $m$ of the Gauss map satisfies $0\leq m \leq2$).
In this section, we will classify complete improper affine fronts, up to transformations in Remark \tr{\ref{equiaffineremark}},  whose total curvatures are $0, -2\pi, -4\pi,$ and $-6\pi$,  \tr{and give a partial classification for the case of $-8\pi$.}

Let $\psi\colon \Sigma=\BSigma_g\setminus \{p_1, \tr{\dots} , p_n\} \to \C \times \R = \R^3$
be a complete improper affine front with Weierstrass data $(F, G)$.
At first, we start from describing three facts \trr{to classify surfaces, which were shown} in \cite {IA-map}.

\begin{fact}{\cite[Theorems 5, 6, 7]{IA-map}}\label{class}
\begin{itemize}
\item
A complete improper affine front is the elliptic paraboloid if and only if its Lagrangian Gauss map $\rho$ in \eqref{gaussmap} is constant.
\item
A complete improper affine front with only one end, which is embedded, is the elliptic paraboloid \tg{in} Example \ref{paraboloid}.
\item
If a complete improper affine front with \tg{exactly} two ends, which are embedded, then it is either the rotational improper affine front or the non-rotational improper affine front described in Examples \ref{R-IA} and \ref{NR-IA}.
\end{itemize}
\end{fact}

The first assertion of Fact \Ref{class} yields the following theorem directly:
\begin{theorem}\label{tc0}
A complete improper affine front with the total curvature $0$ is the elliptic paraboloid.
\end{theorem}

	From now on, without loss of generality, we may assume that $\rho(p)=\infty$ at one end $p\in\overline{\Sigma}$ \tg{by changing roles of $F$ and $G$ and applying a transformation \eqref{equiaffine}}. 
In addition, we \tr{note} the following:

\begin{remark}
\begin{enumerate}\label{bunruityui}
\item
\too{We say that  a meromorphic function $f$ on a compact Riemann surface $\BSigma$ has a zero (resp. pole) of \emph{order} $m\in\Z_{>0}$ 
at $p$, denoting $\ord_p f=m$ (resp. $-m$), if $f$ can be written as $f(z) = (z-p)^m\tilde f(z)$ (resp. $(z-p)^{-m}\tilde f(z)$) locally, where $z$ is a local complex coordinate of $\BSigma$ and $\tilde f$ is a \trr{holomorphic} function around $p$ with $\tilde f(p) \neq 0$.
In particular, we say the zero (resp. pole) is \emph{simple} if $m=1$.  
Elementary complex-analytic arguments give that
$$
\sum_{p\in\BSigma}\ord_p f=0,
\qquad
\deg f = \sum_{p\in f^{-1}(\infty)}(-\ord_p f).
$$
}%
\item
(Residue condition)\ The residues of $F^\pr$ and $G^\pr$ \tg{($^\pr\ \coloneqq d/dz$)} vanish \tr{by the uniqueness of Laurent expansion.}
\item\label{relation2}
If $\rho$ has a pole \tr{of order} \too{$k$} \tr{at a point other than the ends}, then it is a \tr{zero} of $G^\pr$ of order $k$, and $F^\pr\neq0$ holds there because $(F,G)$ satisfies the relation
\begin{equation}\label{relation}
dF = \rho dG,
\end{equation}  
and $(dF,dG)\neq(0,0)$ on $\Sigma$ (cf., Definition \ref{IA-mapdef}).
In particular, $dG\neq0$ on $\Sigma\setminus\rho^{-1}(\{\infty\})$. 
\item
The number of ends $n$ satisfies $n\geq1$. 
In fact, if $n=0$, then $F$ and $G$ \tr{are} holomorphic functions on the compact Riemann \trr{surface} $\BSigma_g$, so they \tr{must} be constant. It contradicts $(dF,dG)\neq(0,0)$.
\end{enumerate}
\end{remark}

\subsection{The case of total curvature $-2\pi$}

\begin{theorem}\label{tc2}
A complete improper affine front with the total curvature $-2\pi$ is obtained from the Weierstrass deta 
\begin{equation}\label{-2pi}
F=az^2, \quad G=z \qquad (a>0)
\end{equation}
defined on $\Sigma=\C$ (Figure \ref{-2pipic}).
\end{theorem}

\begin{proof}
\tr{By} \eqref{ossermanineq2}, $(g,n)$ satisfies
$g+n \leq 3/2.$
Then, we see that $(g,n)=(0,1)$.
We may assume that $\Sigma = \widehat{\C}\setminus \{\infty\} = \C$ and $F, G$ are both polynomials.
Moreover, $\rho$ is a M\"{o}bius transformation \tr{because}  $\deg\rho=1$.
By \eqref{relation},
it holds that $G^\pr=c\ (\neq0)$. 
Hence, we may set $G=z$.
Since \tr{$\rho$} satisfies $\rho(\infty)=\infty$, $\rho$ can be written as 
$\rho=az+b\ (a\neq0)$.
Then, we obtain $F=az^2+bz$. 
\tr{After a suitable} coordinate change and equiaffine transformations \eqref{equiaffine},
we get \eqref{-2pi}.
\end{proof}

\begin{figure}[h]
\begin{center}
\begin{tabular}{c@{\hspace{4cm}}c}
\includegraphics[width=28mm]{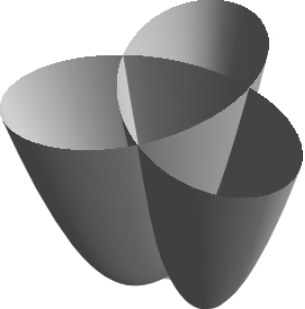}&
\includegraphics[width=28mm]{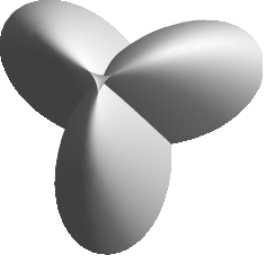}\\
\end{tabular}
\caption{\trr{Complete} improper affine front with total curvature $ -2\pi$ \eqref{-2pi}}\label{-2pipic}
\end{center}
\end{figure}

\begin{remark}
The \cite{krsuy} criteria of singularities for improper affine fronts (\cite{Kodachi19}) shows that the improper affine front in Theorem \ref{tc2} has three swallowtails. A relation between this surface and a flat front in hyperbolic 3-space with three swallowtails is \trr{referred} to in \cite{MM_14} (see also \cite{MM_14_geometric_aspects}). 

\end{remark}

\subsection{The case of total curvature $-4\pi$}

\begin{theorem}\label{tc4}
Complete improper affine fronts with the total curvature $-4\pi$ are the rotational improper affine front, the non-rotational improper affine front, 
and the surfaces constructed by the Weierstrass data
\begin{spacing}{0.95}
\begin{equation}\label{tc-4pi genus-0 1-end 1}
F=az^3+bz, \quad G=z \qquad(a>0, \ b\in\C),
\end{equation}
\begin{equation}\label{tc-4pi genus-0 1-end 2}
F=az^3+bz^2+cz, \quad G=z^2 \qquad (a>0,\  c\in\C\setminus\{0\})
\end{equation}
\end{spacing}
defined on $\Sigma=\C$
(Figure \Ref{tc4pi01}).
\end{theorem}

\begin{proof}
It follows from \eqref{ossermanineq2} that
$g+n \leq 2,$
and the pairs of $(g,n)$ are $(g,n) = (0,2), (1,1), (0,1).$
Recalling Fact \ref{class}, we find that if $(g,n) = (0,2)$, then the surface is either the rotational improper affine front \eqref{RIA} or the non-rotational improper affine front \eqref{NRIA}.
The case of $(g,n) = (1,1)$ \trr{cannot} happen by Fact \ref{class}.
\trr{Hence}, we only have to investigate the case of $(g,n) = (0,1)$.
As \tr{in} the case $\tr{\deg\rho=1}$, we may assume that $\Sigma=\C$. 
Then, $F$ and $G$ are polynomials.

\tr{We} will consider the following two cases.

{\bf (I)}\
The case $\ord_{\infty}\rho = -2$.

We \tr{may} set $G=z$ and $\rho = a_2z^2+a_1 z+a_0\ (a_2\neq0, \ a_1,a_0\in\C)$ by the same reason as in the proof of Theorem \ref{tc2}. Thus, $F$ can be computed, and we obtain
\eqref{tc-4pi genus-0 1-end 1}.

{\bf (II)}\
\tg{The case $\ord_\infty\rho=-1$}

\tg{In this case, there exists unique $p\in \C$ such that $\ord_p\rho=-1$.}

Without loss of generality, we may assume $p=0$.
Then, we have $G^\pr = a z \ (a \neq 0)$.
Hence, $\rho$ and $G$ can be written as
$
G=z^2, \  \rho = a_1 z +a_{-1}/z +a_0 \  (a_1, a_{-1} \in \C\setminus\{0\},\  a_0\in\C),
$
Thus, rewriting the parameters of $\rho$ and changing coordinate,
we have the Weierstrass data \eqref{tc-4pi genus-0 1-end 2}.

Therefore, the proof is completed.
\end{proof}

\begin{figure}[h]
\begin{center}
\begin{tabular}{c@{\hspace{4cm}}c}
\includegraphics[height=28mm]{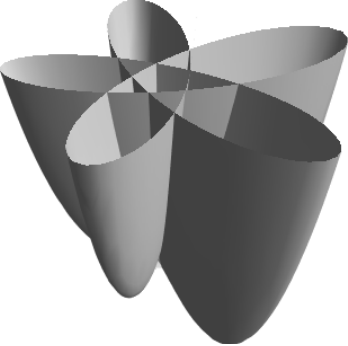}&
\includegraphics[height=28mm]{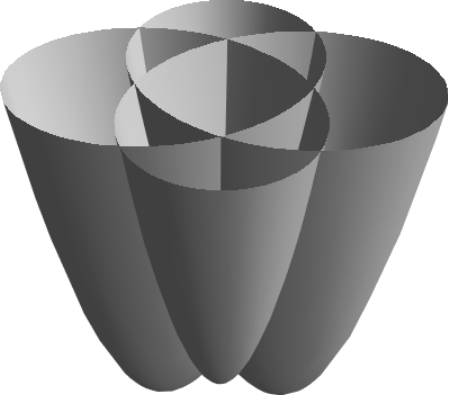}\\
(a) (\Ref{tc-4pi genus-0 1-end 1})
 & (b) (\Ref{tc-4pi genus-0 1-end 2})
\end{tabular}
\end{center}
\caption{Complete improper affine front with total curvature $-4\pi$}\label{tc4pi01}
\end{figure}

\subsection{The case of total curvature $-6\pi$}\label{6pi}

\begin{theorem}\label{tc6}
Complete improper affine fronts with the total curvature $-6\pi$ are constructed by the Weierstrass data 
\begin{equation}\label{601_no1}
F=az^4+bz^2+cz,
\quad
G=z
\quad
(a>0),
\end{equation}
\begin{equation}\label{601_no2}
F=az^4+bz^3+cz^2+dz,
\quad
G=z^2
\quad
(a>0,\ c\neq 0),
\end{equation}
\begin{equation}\label{601_no3}
F=az^4+bz^3+cz^2+dz,
\quad
G=z^3
\quad
(a>0,\ d\neq0),
\end{equation}
\begin{equation}\label{601_no4}
F=az^4+bz^3+cz^2+dz,
\quad
G=\al(2z^3-3z^2)
\quad
(a>0, \ d,\al\neq 0),
\end{equation}
which are defined on $\Sigma=\C$ (Figure \ref{tc6pi01}), 
\begin{equation}\label{602_no1}
F=az^2+bz+\fr{c}{z},
\quad
G=\fr{1}{z}
\quad
(a>0, \ b\in \R,\  |c| \neq 1),
\end{equation}
\begin{equation}\label{602_no2}
F=az + \fr{b}{z}+ \fr{c}{z^2},
\quad
G=\fr{1}{z^2}
\quad
(a>0,\  |c| \neq 1),
\end{equation}
\begin{equation}\label{602_no3}
F= \fr{a}{z^2}+\fr{b}{z} + cz,
\quad
G=\fr{1}{z}
\quad
(a>0, \ c\in\R\setminus\{0\}),
\end{equation}
\begin{equation}\label{602_no4}
F=az+\fr{b}{z}+\fr{c}{z^2},
\quad
G=\alpha\left(\fr{1}{2z^2}-\fr{1}{z}\right)
\quad
(a>0, \ \alpha \in\C\setminus\{0\},\  2|c|\neq |\al|),
\end{equation}
\begin{equation}\label{602_no6}
F=az^2+bz+\fr{c}{z},
\quad
G = \al\left(z+\fr{1}{z}\right)
\quad(a>0,\ b,\al\neq0,\  c-a \in\R,\  |c|\neq |\al|),
\end{equation}
which are  defined on $\tr{\Sigma= \C\setminus\{0\}}$ (Figure\tr{s} \ref{tc6pi01}, \ref{tc6pi02} and \ref{tc6pi02_2}).
\end{theorem}

\noindent
\emph{\tr{P}roof of Theorem \ref{tc6}.}
\tr{By} \eqref{ossermanineq2}, we find
$g+n\leq5/2$, and hence $(g,n) = (0,2), (1,1), (0,1).$
These cases are the same $(g,n)$ as the case of the total curvature $-4\pi$ , but note that at least one end is not embedded in the case $(g,n)=(0,2).$

\underline{{\bf Case 1}}
\quad
$(g,n) = (0,1)\quad (\Sigma = \C).$

We further divide {\bf Case 1} into the following {\bf (I)\tg{--}(III)}.

{\bf (I)}
\
The case $\ord_\infty \rho = -3.$

In this case, as with {\bf (I)} in the case of $-4\pi$ and $(g,n)=(0,1)$, we obtain \eqref{601_no1}.

{\bf (II)}
\
The case where there uniquely exists $p\in \C$ which is a pole of $\rho$.

We \tr{must} consider the two more cases:
$$
\text{(II-a)}\ \ord_p\rho=-1,\quad \ord_{\tr{\infty}}\rho=-2,
\quad
\tg{\text{or}}
\quad
\text{(II-b)}\ \ord_p\rho=-2,\quad \ord_{\tr{\infty}}\rho=-1.
$$

(II-a)
We can set $G^\pr=\alpha(z-p)$, and then after some transformations, we get $G=z^2$.
Since $\rho$ is expressed as $\rho=a_2z^2+a_1z+a_{-1}/z + a_0 \ (a_2, a_{-1}\neq 0)$, computing $F$ from the relation (\Ref{relation}), and conducting some transformations, we have Weierstrass  data \eqref{601_no2}.

(II-b)
Similarly, as (II-a), we can set $G=z^3$, and $\rho$ is written by the formation $\rho = a_{-2}/z^2+a_{-1}/z+a_0+a_1z \ (a_{-2}, a_1 \neq 0)$.
Then, we obtain \eqref{601_no3} by \eqref{relation}.

{\bf (III)}
\
The case where there exist distinct points $p,q\in\C$ which are \tr{simple} poles of $\rho$.

Without loss of generality, we can set $p=0, q=1$.
Then, $G^\pr$ can be written as $G^\pr = \al z(z-1) \ (\al\neq 0)$, and by retaking the parameter $\al$,
we have $G=\al(2z^3-3z^2).$
Therefore, by (\Ref{relation}), we get \eqref{601_no4}.

\begin{figure}[h]
\begin{center}
\begin{tabular}{c@{\hspace{0.9cm}}c@{\hspace{0.9cm}}c@{\hspace{0.9cm}}c}
\includegraphics[height=22mm]{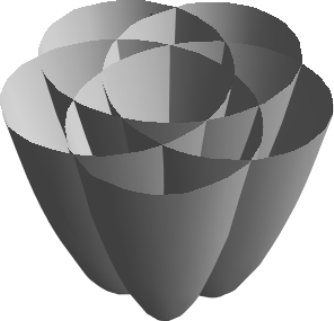}&
\includegraphics[height=22mm]{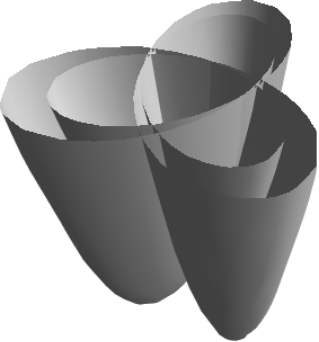}&
\includegraphics[height=22mm]{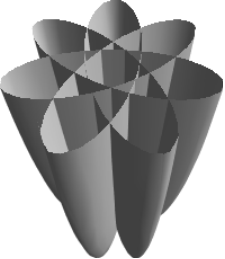}&
\includegraphics[height=19mm]{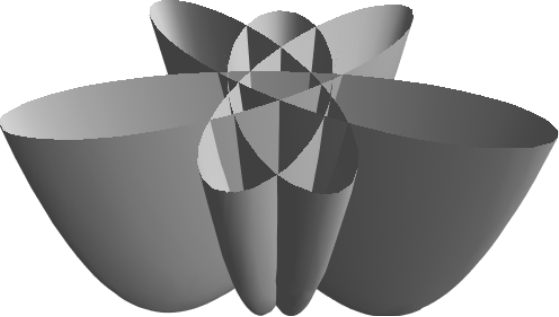}\\
(a) (\Ref{601_no1}) & (b) (\Ref{601_no2}) & (c) (\Ref{601_no3}) & (d) (\Ref{601_no4})
\end{tabular}
\end{center}
\caption{\trr{Complete} improper affine front with total curvature $-6\pi$, and one end} \label{tc6pi01}
\end{figure}

\vskip\baselineskip

\underline{{\bf Case 2}}
\quad
$(g,n) = (0,2) \quad (\Sigma = \C\setminus\{0\}).$

Note that at least one of $F$ and $G$ must not be a polynomial \tr{because of \eqref{relation2} in Remark \ref{bunruityui}}.

\tr{To consider} this case, we must note the period condition and completeness.
We divide this case into the following {\bf (I)\tg{--}(III)}.

{\bf (I)}
\
The case $\ord_{\infty} \rho =-3.$

We can set $\rho =a_3z^3+a_2z^2+a_1z+a_0 \ (a_3\neq 0)$, and $G, G^\pr$ must have a pole at $z=0$.
\tr{By} \eqref{relation}, we \tr{may} set $G^\pr = \al/z^{k} \ (\al\neq 0, k = 2, 3)$, and $G$ is calculated as $G = 1/z^{k-1}$.
Hence after \tr{a suitable} transformation, $F$ is calculated, and \tr{by the period condition \eqref{period}}, we have \eqref{602_no1} and \eqref{602_no2}.

\begin{figure}[h]
\begin{center}
\begin{tabular}{c@{\hspace{4cm}}c}
\includegraphics[width=24mm]{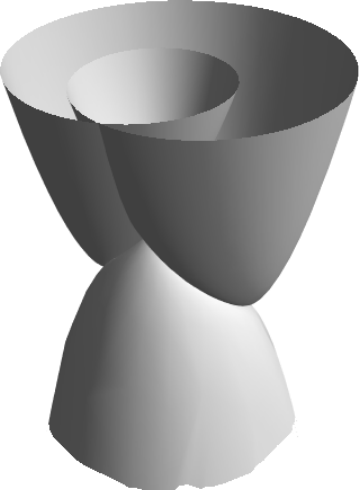}&
\includegraphics[width=24.8mm]{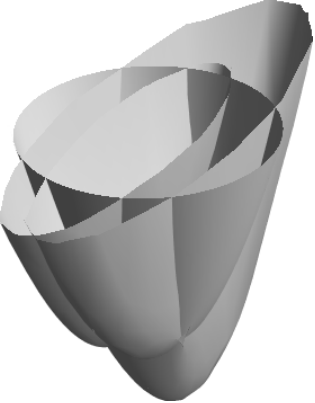}\\
(a) (\Ref{602_no1}) & (b) (\Ref{602_no2})
\end{tabular}
\end{center}
\caption{Complete improper affine fronts with total curvature $-6\pi$, and two ends, No.1}\label{tc6pi02}
\end{figure}

{\bf (II)}
\
The case where there uniquely exists $p\in \C$ which is a pole of $\rho$.

We \tr{must} consider the following two cases:
$$
\text{(II-a)}\ \ord_p\rho=-1,\quad \ord_\infty\rho=-2,
\quad
\tg{\text{or}}
\quad
\text{(II-b)}\ \ord_p\rho=-2,\quad \ord_\infty\rho=-1.
$$

(II-a)
In this case, $\rho$ can be expressed as
$
\rho(z) = a_{-1}/(z-p)+a_0+a_1z+a_2z^2
\ 
(a_{-1},a_2 \neq 0),
$
and $G^\pr \neq 0$ on $\C\setminus\{0, p\}$ \tr{because of \eqref{relation2} in Remark \ref{bunruityui}}. 
Moreover, we divide (II-a) into two more cases.\\
(II-a-1)
When $p=0$, we \tr{may} set $G^\pr = \al/z^2$, so we obtain $G=1/z$.
Thus, we have \eqref{602_no3}.

(II-a-2)
When $p\neq0$\tr{, we may assume $p=1$ and set}
$G^\pr=\al(z-1)/z^3.$
Then,
$a_1-a_2=0$ must hold because of the residue condition for $F^\pr$.
Therefore, we get \eqref{602_no4}.

(II-b)
\ 
$\rho$ is expressed by
$
\rho(z) = a_{-2}/(z-p)^2+a_{-1}/(z-p)+a_0+a_1z
\ 
(a_{-2}, a_1\in\C\setminus\{0\}).
$
Moreover, we divide (II-b) into two more cases.

(II-b-1)
\ 
When $p=0$,
since $G^\pr=\al \ (\al\neq 0)$, we \tr{may} set $G=z$.
Given the residue condition for $F^\pr$, we have
$
F=a/z+bz+cz^2.
$
However, by changing coordinates, we find that this data is the same as \eqref{602_no3}.

(II-b-2)
When $p\neq 0$,
we \tr{may} set
$
G^\pr = \al(z-p)^2/z^k
\ 
(\al \neq 0),
$
and $k$ satisfies $k\geq 4$. 
Then, one can \trr{verify} that $F$ does not have a pole at $\infty$. Thus, this \tr{case} is impossible.

{\bf (III)}
\
The case where there exist distinct points $p,q\in\C$ which are \tr{simple} poles of $\rho$.

$\rho$ is expressed as
$
\rho = a_{-1}/(z-p)+b_{-1}/(z-q)+a_0+a_1z
\ 
(a_{-1}, b_{-1}, a_1 \neq 0).
$
We will consider the following two cases:

(III-a)
\
When $p=0$,
$G^\pr(q)=0$ of order $1$.
Then, we can set
$
G^\pr=\al(z-q)/z^k
\ 
(\al\neq 0, k\geq3).
$
However $\ord_{\infty} F^\pr=\ord_{\infty}\rho+\ord_{\infty}G^\pr = -1+(k-1)=k-2\geq1$, and this \tr{case} is impossible because both $F$ and $G$ do not have a pole at $\infty$.

(III-b)
\
When $p, q \neq 0$ (we \tr{may} set $p=1$),
$G^\pr$ can be written as
$
G^\pr = \al(z-1)(z-q)/z^k
\ 
(\al \neq 0 ),
$
and $k$ satisfies $k=2$ if $1+q=0$, or $k\geq 4$.
However, the latter \tr{case} is impossible for the same reason as (III-a).
If $k=2$, then
we obtain
$G = \al(z+1/z),\ \rho = a_{-1}/(z-1) + b_{-1}/(z+1) + a_1 z + a_0\ (a_{-1}, b_{-1}, a_1 \neq0).$
Retaking parameters, we obtain \eqref{602_no6}.

\begin{figure}[h]
\begin{center}
\begin{tabular}{c@{\hspace{2.5cm}}c@{\hspace{2.5cm}}c}
\includegraphics[width=24mm]{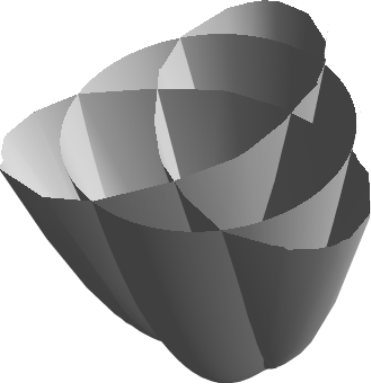}&
\includegraphics[width=22mm]{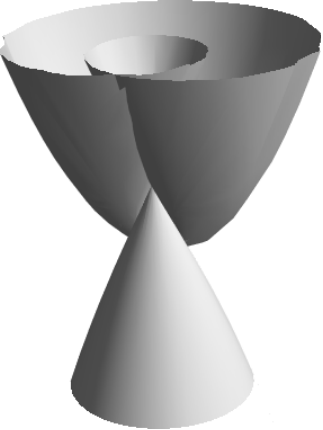}&
\includegraphics[width=22mm]{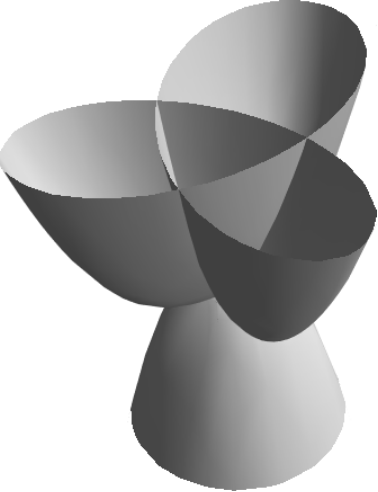}\\
(a) (\Ref{602_no3}) & (b) (\Ref{602_no4}) & (c) (\Ref{602_no6})
\end{tabular}
\end{center}
\caption{Complete improper affine fronts with total curvature $-6\pi$, and two ends, No.2}\label{tc6pi02_2}
\end{figure}

\underline{{\bf Case 3}}{
\quad
$(g,n) = (1,1).$

Let $\tau$ be a complex number satisfying $\Im{\tau} >0$, and set
\begin{equation}\label{torus}
T_\tau\coloneqq \C/[1,\tau],
\end{equation}
where $[1,\tau]$ is a lattice defined by $[1,\tau] \coloneqq \{m + n\tau \ ;\ m,n \in \Z\}$.
In this case, we may assume that $\Sigma=T_\tau \setminus \{[0]\}$ ($[x]$ stands for \tr{the} equivalence class \tr{of} $x$). 
Since $F$,$G$, and $\rho$ are meromorphic functions on $T_\tau$, we can identify them with the elliptic functions on $\C$ associated with $[1,\tau]$. 
We will \tr{apply} general theory of the elliptic functions when we consider the case of genus $1$ (see \cite{courant} for \trr{details}). Let $\Pi_0 \coloneqq \{x + y\tau \ ;\ 0\leq x,y <1\}$ be a fundamental period parallelogram (FPP, in short).
Here, the Weierstrass $\wp$-function associated \tg{with} $[1,\tau]$ is defined by
$$
\wp(z)\coloneqq \fr{1}{z^2}+\sum_{\omega\in[1,\tau],\omega\neq0}\left(\fr{1}{(z-\omega)^2}-\fr{1}{\omega^2}\right), 
$$
and the Laurent expansion of $\wp$ around $0$ is 
$$
\wp(z)
=
\fr{1}{z^2} + 3G_4z^2 + 5G_6z^4 + \ddd,
$$
where $G_k \coloneqq \sum_{\omega \in [1,\tau]\setminus \{0\}}(1/\omega^k)\ (k=4,5,\tr{\dots})$.

\begin{fact}[\cite{courant}]\label{wp}
\begin{enumerate}
\item
Set $e_j \coloneqq \wp(\omega_j)\ (j=1,2,3),$
where
$
\omega_1 \coloneqq 1/2,
\omega_2 \coloneqq (1+\tau)/2,
\omega_3 \coloneqq \tau/2,  
g_2 \coloneqq 60G_4$, and $g_3\coloneqq 140G_6$. 
Then, $\wp$ satisfies
$$
\tr{\wp^\pr(z)^2} = 4\wp(z)^3-g_2\wp(z)-g_3,
$$
\begin{equation}\label{g2g3}
g_2=-4(e_1e_2+e_2e_3+e_3e_1),
\quad
g_3=4e_1e_2e_3.
\end{equation}
\item
\trr{Every} non-constant elliptic function has at least one pole in the FPP. 
In other words, \trr{any} elliptic function which is holomorphic on the FPP is constant.
\item
The sum of residues in the FPP of an elliptic function is $0$.
\item
The number of zero points in the FPP of an elliptic function is equal to the number of its poles in the FPP.
Here, the number of zero points (resp\tr{.} poles) is the sum of the order at each zero point (resp\tr{.} pole).
\end{enumerate}
\end{fact}

So we divide {\bf Case 3} into the following {\bf (I)\tg{--}(III)}.

{\bf (I)}
\
The case $\ord_0 \rho = -3$.

Here, $G^\pr \neq 0$ holds on $\Pi_0$ by \eqref{relation}.
Since $G^\pr$ is holomorphic in $\Pi_0$, $G^\pr(z) = c\ (\neq 0)$ holds, and then $G(z)=cz$, but this is not an elliptic function.
Thus, {\bf (I)} does not happen.

{\bf (II)}
\
The case where there uniquely exists $p\in\Pi_0\setminus\{0\}$ which is a pole of $\rho$.

We \tr{must} consider two more \trr{cases:}
$$
\text{(II-a)}\ \ord_p\rho=-1,\quad \ord_0\rho=-2,
\quad
\tg{\text{or}}
\quad
\text{(II-b)}\ \ord_p\rho=-2,\quad \ord_0\rho=-1.
$$

(II-a)
In this case, 
$G^\pr(p)=0$ of order $1$ and $G^\pr\neq 0$ otherwise.
Then, $z=0$ is the only pole of $G^\pr$ of \too{order $1$}, and thus $\Res(G^\pr,z=0) \neq 0$.
This is impossible.

(II-b)
In this case, $z=0$ is the only pole of $G^\pr$ of \too{order $2$}, and then $z=0$ is the only pole of $G$ of \too{order $1$}.
This \trr{case} is impossible \tr{by} the same reason as (II-a).

{\bf (III)}
\
The case where there exist distinct points $p,q\in\Pi_0\setminus\{0\}$ which are \tr{simple} poles of $\rho$.

Then, since $G^\pr(p) = G^\pr(q) = 0$ of order $1$, $G^\pr$ has the unique pole at $z=0$ of \too{order $2$}. 
Hence, this is impossible for the same reason as (II-b).

Summing up the arguments above, we find that there \trr{do} not exist complete improper affine fronts of genus $1$ with total curvature $-6\pi$.
$\square$

\subsection{The case of total curvature $-8\pi$}

In this case, \trr{by the} Osserman-type inequality \eqref{ossermanineq2}, we obtain
$g+n \leq 3$
and find $(g,n)=(0,1),(0,2), (1,1),(1,2), (2,1),(0,3)$.
The cases $(g,n) = (1,2), (2,1)$ \trr{cannot} happen by Fact \ref{class}.
\begin{theorem}\label{tc80}
Complete improper affine fronts with genus $0$ whose total curvature is $-8\pi$ are constructed by the Weierstrass data
\begin{equation}\label{801_no1}
F= az^5+bz^3+cz^2+dz,
\quad
G=z
\quad
(a>0),
\end{equation}
\begin{equation}\label{801_no2}
F=az^5+bz^4+cz^3+dz^2+ez,
\quad
G=z^2
\quad(a>0, \ e\in\C\setminus\{0\}),
\end{equation}
\begin{equation}\label{801_no3}
F=az^5+bz^4+cz^3+dz^2+ez,
\quad
G=z^3
\quad
(a>0, \ e\in\C\setminus\{0\}),
\end{equation}
\begin{equation}\label{801_no4}
F=az^5+bz^4+cz^3+dz^2+ez,
\quad
G=z^4
\quad
(a>0, \ e\in\C\setminus\{0\}),
\end{equation}
\begin{equation}\label{801_no5}
F=az^5+bz^4+cz^3+dz^2+ez,
\quad
G=\al(2z^3-3z^2)
\quad
(a>0, \ \al, e, F^\pr(1)\neq0),
\end{equation}
\begin{equation}\label{801_no6}
F=az^5+bz^4+cz^3+dz^2+ez,
\quad
G=\al(3z^4-4z^3)
\quad
(a>0, \ \al, e, F^\pr(1)\neq 0),
\end{equation}
\begin{equation}\label{801_no7}
F=az^5+bz^4+cz^3+dz^2+ez,
\quad
G=\al(3z^4-4(1+r)z^3+6rz^2)
\end{equation}
$(a>0,\  r\notin\{0,1\},\  \al, e, F^\pr(1), F^\pr(r)\neq 0)$,
\tr{which are} defined on $\Sigma=\C$ \tg{(Figure \ref{tc8pi02})}, 
\begin{align}
F=az^3+bz^2+cz+\fr{d}{z},
\quad
G=\fr{1}{z}
\quad(a>0, c\in\R, |d|\neq1),\label{802_no1}\\
F=az^2+bz+\fr{c}{z}+\fr{d}{z^2},
\quad
G=\fr{1}{z^2}
\quad
(a>0, |d|\neq1),\label{802_no2}
\end{align}
\begin{align}
F=az^2+bz+\fr{c}{z}+\fr{d}{z^2},
\quad
G=\fr{1}{z}
\quad
(a>0, d\neq0),\label{802_no3}\\
F=az+\fr{b}{z}+\fr{c}{z^2}+\fr{d}{z^3},
\quad
G=\fr{1}{z^2}
\quad
(a>0, d\neq0),\label{802_no4}
\end{align}
\begin{equation}\label{802_no5}
F=az^2+bz+\fr{c}{z}+\fr{d}{z^2},
\quad
G=\al\left(-\fr{1}{z}+\fr{1}{2z^2}\right)
\end{equation}
$(a>0,\ b\in\R,\ \al, F^\pr(1)\neq0, \ 2|d|\neq|\al|),$
\begin{equation}\label{802_no6}
F=az+\fr{b}{z}+\fr{c}{z^2}+\fr{d}{z^3},
\quad
G=\al\left(-\fr{1}{2z^2}+\fr{1}{3z^3}\right)
\end{equation}
$(a>0, \ F^\pr(1)\neq0,\ 3|d|\neq|\al|),$
\begin{equation}\label{802_no7}
F=az+\fr{b}{z}+\fr{c}{z^2}+\fr{d}{z^3},
\quad
G=\fr{1}{z}
\quad
(a>0, \ d\neq0),
\end{equation}
\begin{equation}\label{802_no8}
F=az+\fr{b}{z}+\fr{c}{z^2}+\fr{d}{z^3},
\quad
G=\al\left(-\fr{1}{z}+\fr{1}{z^2}-\fr{1}{3z^3}\right)
\end{equation}
$(a>0, \ F^\pr(1)\neq0,\  3|d|\neq|\al|),$
\begin{equation}\label{802_no9}
F=az+\fr{b}{z}+\fr{c}{z^2}+\fr{d}{z^3}
\quad
G=\al\left(\fr{1}{2z^2}-\fr{1}{z}\right)
\quad
(a>0, \ \al, d\neq0),
\end{equation}
\begin{equation}\label{802_no10}
F=az^3+bz^2+cz+\fr{d}{z},
\quad
G=\al\left(z+\fr{1}{z}\right)
\end{equation}
$
(a>0, \ \al,F^\pr(\pm1)\neq0, \ -c+d\in\R, \ |d|\neq|\al|),
$
\begin{equation}\label{802_no11}
F=az+\fr{b}{z}+\fr{c}{z^2}+\fr{d}{z^3}
\quad
G=\al\left(-\fr{1}{z}+\fr{p+1}{2z^2}-\fr{p}{3z^3}\right)
\end{equation}
$(a>0,\ \al\in\R,\  p\in\C\setminus\{0,1\},\ \al, F^\pr(1), F^\pr(p)\neq0,\  3|d|\neq|\al||p|),$
\begin{equation}\label{802_no12}
F=az^3+bz^2+cz+\fr{d}{z},
\quad
G=\al\left(z^2-6z+\fr{8}{z}\right)
\end{equation}
$(a>0, \ \al(3d+4c)\in\R, \ \al, F^\pr(1), F^\pr(-2)\neq0,\  |d|\neq 8|\al|),$
\begin{equation}\label{802_no13}
F=az^2+bz+\fr{c}{z}+\fr{d}{z^2},
\quad
G=\al\left(z+\fr{3}{4z}+\fr{1}{8z^2}\right)
\end{equation}
$(a>0,\  \al(a+3b-4c)\in\R, \ F^\pr(1), F^\pr(-1/2),\al\notin\C\setminus\{0\}, \ 2|d|\neq |\al|),$
\begin{equation}\label{802_no14}
F=az^2+bz+\fr{c}{z}+\fr{d}{z^2},
\quad
G=\al\left(z+\fr{1}{z}\right)
\end{equation}
$(a>0,\  \al, d,F^\pr(\pm1)\notin\C\setminus\{0\},\  \al(b-c)\in\R),$
\begin{align}\label{802_no15}
F=az^3+bz^2+cz+\fr{d}{z},
\quad
G=\al\left(z^2+2(pq-1)z+\fr{2pq}{z}\right)\\\nonumber
\tg{\left(
\begin{aligned}
& a>0, \; p\neq q, \; p,q\notin\{0,1\}, \; p+q=-pq, \; (d-c)pq-d\in\R,\\
& |d|\neq2|pq|, \; F^\pr(p),  F^\pr(q),  F^\pr(1)\notin\C\setminus\{0\}
\end{aligned}
\right),
}
\end{align}
\begin{align}\label{802_no16}
F=az^2+bz+\fr{c}{z}+\fr{d}{z^2},
\quad
G=\al\left(z+\fr{q^2+q+1}{z}-\fr{q(q+1)}{2z^2}\right)\\ \nonumber
\tg{
\left(
\begin{aligned}
&a>0,\  \al\neq0, \ q\notin \{0,\pm1\}, \ \al(4c-4b(q^2+q+1)+aq(q+1))\in\R,\\
&4|d|\neq|\al q(q+1)|,\ F^\pr(1), F^\pr(q),F^\pr(-1-q)\notin\C\setminus\{0\}
\end{aligned}
\right), 
}
\end{align}
\tr{which are} defined on $\C\setminus\{0\}$ \tg{(Figure \ref{tc8pi02})},
and
\begin{equation}\label{803_no1}
F=az+\fr{b}{z-1}+\fr{c}{z},
\quad
G=\fr{\al}{z-1}
\quad
(a>0, \ c,\al\in\R\setminus\{0\},\  |b|\neq\al)
\end{equation}
\begin{equation}\label{803_no2}
F=az+\fr{b}{z-1}+\fr{c}{z},
\quad
G=\al\left(\fr{1}{z}-\fr{1}{z-1}\right),
\end{equation}
$(a>0, \ \Im(b+c)=0,\  \al\in\C\setminus\{0\},\  |b|,|c|\neq|\al|),$
\begin{align}\label{803_no3}
F=az+\fr{b}{z-1}+\fr{c}{z}, \quad
G=\al\left(\fr{pq-1}{z-1}-\fr{pq}{z}\right),\hspace{3cm}\\ \nonumber
\tg{
\left(
\begin{aligned}
&a>0,\  \al\in\C\setminus\{0\}, \ p,q\notin\{0,1\},\ pq\neq1, \ p+q=2pq, \ \Im(1-2pq)=0, \\
&\Im(2(b-c)pq+2c-a)=0,\ |c|\neq|\al||pq|,\  |b|\neq|\al||pq-1|,
\end{aligned}
\right)
}
\end{align}

\tr{which are} defined on $\Sigma=\C\setminus\{0,1\}$ 
(Figure \ref{tc8pi03}).
\end{theorem}

\begin{proof}
\tr{As the same way as the previous cases} of $(g,n)=(0,1), (0,2)$, we get the above Weierstrass data.
Hence, we only need to consider the case of $(g,n) = (0,3).$

In this case, the three ends are all embedded.
We may assume that $\Sigma = \C \setminus\{0,1\}$ and $\rho(\infty) = \infty$.
Firstly, we shall divide this case into the following {\bf (I)\tg{--}(IV)}.

{\bf (I)}
\
The case $\ord_\infty\rho=-4$.

By (\Ref{relation}), we find $G^\pr \neq 0$ on $\C$, and observe that
$
(\ord_\infty G^\pr, \ord_0 G^\pr, \ord_1G^\pr) = (4,-2,-2).
$
Hence, $G^\pr$ can be expressed as
$
G^\pr = \al/z^2(z-1)^2
\ 
(\al \neq0).
$
However, since $\Res(G^\pr, 0)=2\al\neq0$, this case does not happen.

{\bf (II)}
\
The case where there uniquely exists $p\in\C$ which is a pole of $\rho$.

In addition, we must investigate two cases (II-a)\ $p\in\{0,1\}$ and (II-b)\ otherwise.

(II-a)
\
The case $p\in\{0,1\}$.

It is sufficient to consider the case of $p=0$. 
Moreover, this case has to be divided into the following two cases by the orders of $\rho$: 

(II-a-1)
\ 
The case $(\ord_\infty \rho,\ord_0\rho)=(-1,-3).$

We observe
$(\ord_\infty G^\pr, \ord_0 G^\pr, \ord_1G^\pr) = (4,-2,-2),\ (2,0,-2),$ $(0,2,-2).$
The first case is the same as {\bf (I)}, so \tr{this case is} impossible.
The second \tr{case} is also impossible by 
$
\ord_{\infty}F^\pr=\ord_{\infty}\rho+\ord_{\infty}G^\pr=-1+2=1.
$
\trr{In} the last case,
$G^\pr$ can be written as
$
G^\pr=\al z^2/(z-1)^2
\ 
(\al\neq0).
$
Then, we find $\Res(G^\pr,1)=2\al\neq0$, and this is \tr{again} impossible.

(II-a-2)
\ 
The case $(\ord_\infty\rho,\ord_0\rho)=(-2,-2)$.

One can find $(\ord_\infty G^\pr, \ord_0G^\pr, \ord_1G^\pr)=(4,-2,-2), (2,0,-2), (0,2,-2).$
We find $\Res(G^\pr,1)\neq0$ in the first and second cases, so they are \trr{impossible}. 
In the third case, we obtain 
$$
G=\fr{\al}{z-1}
\quad 
(\al\in\C\setminus\{0\}).
$$
\tr{By} \eqref{relation}, we get the Weierstrass \trr{data} \eqref{803_no1}.

(II-b)
\
The case $p\notin\{0,1\}.$

Furthermore, this case is divided into the following three cases:

(II-b-1)
\
The case $(\ord_\infty \rho,\ord_p\rho)=(-3,-1).$

In this time,
we observe $\ord_pG^\pr=1, \ord_0G^\pr=\ord_1G^\pr=-2$, and 
$\ord_\infty G^\pr=3$.
Hence, we can write $G^\pr$ as
$$
G^\pr = \fr{\al(z-p)}{z^2(z-1)^2}
\quad
(\al \in\C\setminus\{0\}),
$$
and by $\Res(G^\pr, 0)=\al(1-2p), \Res(G^\pr, 1)=\al(2p-1)$, $p$ must be $p=1/2$.
Thus, we obtain
$$
G=\al\left(\fr{1}{z}-\fr{1}{z-1}\right)
\quad
(\al\neq0).
$$
We rewrite $\al/2$ as the same $\al$.
Then, $F$ can be calculated by \eqref{relation}, and
we get the Weierstrass  data \eqref{803_no2}.

(II-b-2)
\
The case $(\ord_\infty \rho,\ord_p\rho)=(-2,-2).$

In this case, the orders of $G^\pr$ satisfies
$
(\ord_\infty G^\pr, \ord_0G^\pr, \ord_1G^\pr, \ord_pG^\pr)=(2,-2,\\-2,2).
$ 
Hense, $G^\pr$ can be expressed as
$
G^\pr = \al(z-p)^2/z^2(z-1)^2
\ 
(\al\neq0).
$
Then, it \too{holds} that $\Res(G^\pr,0)=2\al p(p-1)$. Thus $p\neq 0,1$ yields a contradiction.

(II-b-3)
\
The case $(\ord_\infty \rho,\ord_p\rho)=(-1,-3).$

Then,
$
(\ord_\infty G^\pr, \ord_0G^\pr, \ord_1G^\pr, \ord_pG^\pr)=(1,-2,-2,3),
$
and we can express $G^\pr$ as
$
G^\pr=\al(z-p)^3/z^2(z-1)^2
\ 
(\al\neq0).
$
It is necessary to hold $\Res(G^\pr,0)=\al p^2(2p+3)=0$ and $\Res(G^\pr,1)=\al(p-1)^2(2p+1)=0$, but
it is impossible for $p\neq 0,1$.

{\bf (III)}
\
The case where there \trr{exist} distinct points $p,q\in\C$ which are poles of $\rho$.

We will check the three cases
(III-a) $p=0, q=1$, (III-b) $q=0$, $p\notin \{0,1\}$, and (III-c) $p,q \notin \{0,1\}$.
The first case (III-a) is impossible for the same reason as (II-a).

(III-b)
\
The case $q=0$, $p\notin \{0,1\}$.

In this case, the orders of $\rho$ are
$$
(\ord_\infty\rho,\ord_0\rho,\ord_p\rho) = (-2,-1,-1),\tr{\text{or}} \ (-1,-2,-1), \tr{\text{or}}\ (-1,-1,-2),
$$
and $\ord_1G^\pr =-2$ must hold.

(III-b-1)
\
The case $(\ord_\infty\rho,\ord_0\rho,\ord_p\rho) = (-2,-1,-1)$.

Since
$\ord_pG^\pr =1, \ord_0G^\pr \geq1$, $ \ord_\infty G^\pr \geq2$, and $\ord_1G^\pr=-2$, we \tr{see} that
$\sum_{x\in\widehat{\C}}\ord_xG^\pr \geq 1+1+2-2=2$. This is impossible.

(III-b-2)
\
The case $(\ord_\infty\rho,\ord_0\rho,\ord_p\rho)=(-1,-2,-1).$

We know that
$
(\ord_\infty G^\pr, \ord_0G^\pr, \ord_1G^\pr, \ord_pG^\pr)=(1,-2,0,1),\tr{\text{or}}\ (1,0,-2,1).
$
In the first case, we find $\Res(G^\pr, 0)\neq0$. In the second case, we also find $\Res(G^\pr, 1)\neq0$. Hence, these cases are impossible.

(III-b-3)
\
The case $(\ord_\infty\rho,\ord_0\rho,\ord_p\rho) =(-1,-1,-2).$

Given $\ord_pG^\pr=2, \ord_\infty G^\pr=1,$ and $\ord_1G^\pr=-2$, we \tr{see that} $\ord_0G^\pr=-1$.
However, this is impossible.

(III-c)
\
The case $p,q\notin \{0,1\}$.

Furthermore, this case is divided into the following two cases:
$$
(\ord_\infty\rho,\ord_p\rho,\ord_q\rho) = (-2,-1,-1),  \tr{\text{or}}\ (-1,-2,-1).
$$

(III-c-1)
\
The case $(\ord_\infty\rho,\ord_p\rho,\ord_q\rho) = (-2,-1,-1)$.

We observe
$
(\ord_\infty G^\pr, \ord_0G^\pr, \ord_1G^\pr, \ord_pG^\pr, \ord_qG^\pr)=(2,-2,-2,1,1)
$
and \tr{may} set
$$
G^\pr
=
\fr{\al(z-p)(z-q)}{z^2(z-1)^2}\\
=
\al\left(
\fr{pq-p-q+1}{(z-1)^2}+\fr{pq}{z^2}+\fr{p+q-2pq}{z-1}+\fr{2pq-p-q}{z}
\right).
$$
Then, $p+q=2pq, \ pq\neq0,1$ must hold.
Hence, we obtain
$$
G=\al\left(\fr{pq-1}{z-1}-\fr{pq}{z}\right),
$$
and by \eqref{relation}, we get  the Werstrass data \eqref{803_no3}.

(III-c-2)
\
The case $(\ord_\infty\rho,\ord_p\rho,\ord_q\rho) = (-1,-2,-1)$. 

In the same way \trr{as} the case of (III-c-1), we have
$
(\ord_\infty G^\pr, \ord_0G^\pr, \ord_1G^\pr,\\ \ord_pG^\pr, \ord_qG^\pr)=(1,-2,-2,2,1).
$
Then, we obtain
$$
G^\pr=\al\left(-\fr{(q-1)(p-1)^2}{(z-1)^2}-\fr{p^2q}{z^2}-\fr{(p-1)(2pq-p-1)}{z-1}+\fr{p(p+2q-2pq)}{z}\right)
$$
and $2pq-p-1=0$ and $p+2q-2pq=0$ must hold, but \tr{such $p$ and $q$ do not exist.} 

{\bf (IV)}
The case where there exist distinct points $p,q,r\in\C$ which are \tr{simple} poles of $\rho$.

\tr{We} investigate three more cases:

(IV-a)
\
The case $p=0, q=1,$ and $r\notin\{0,1\}.$

Then,
$
\ord_rG^\pr=1,
\ord_0G^\pr,\ord_1G^\pr, \ord_\infty G^\pr\geq1
$
hold.
Hence, we have \\
\noindent$\sum_{x\in\widehat{\C}}\ord_xG^\pr\geq 4$, and it does not happen.

(IV-b)\ The case $p=0,$ and $q,r\notin\{0,1\}$.

It holds that
$\ord_1G^\pr=-2, \ord_qG^\pr=\ord_rG^\pr=1, \ord_0G^\pr\geq0$ and $\ord_{\infty}G^\pr\geq1$.
Thus,
$\sum_{x\in\widehat{\C}}\ord_xG^\pr \geq 1$ holds, and this is a contradiction.

(IV-c)
\
The case $p,q,r\notin\{0,1\}.$

Then, we find that
$
(\ord_\infty G^\pr, \ord_0G^\pr, \ord_1G^\pr, \ord_pG^\pr, \ord_qG^\pr, \ord_rG^\pr) = (1,-2,\\-2,1,1,1).
$
Hence,
$$
G^\pr = \al\left(\fr{a}{(z-1)^2}+\fr{2pqr-pq-qr-rp+1}{z-1}-\fr{b}{z^2}+\fr{pq+qr+rp-2pqr}{z}\right)
$$
$(a, b\in\C, \al\neq0)$ holds, but this does not happen.
In fact, $2pqr-pq-qr-rp+1=0$ and $pq+qr+rp-2pqr=0$ must hold, but \tr{such $p$ and $q$ do not exist.}

Therefore, the proof is finished.
\end{proof}

\begin{figure}[h]
\begin{center}
\begin{tabular}{c@{\hspace{2cm}}c@{\hspace{2cm}}c}
\includegraphics[width=20mm]{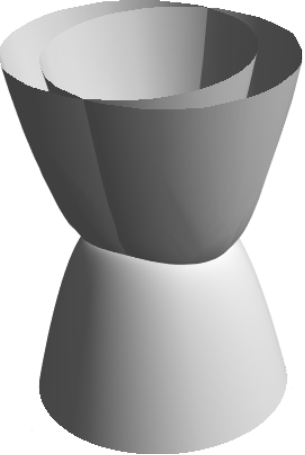}&
\includegraphics[width=20mm]{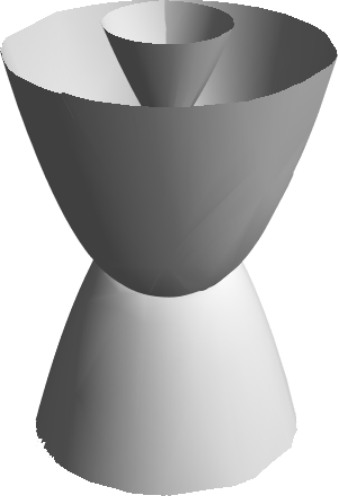}&
\includegraphics[width=20mm]{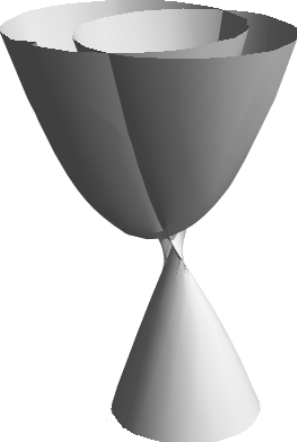}\\
(a) (\Ref{803_no1}) & (b) (\Ref{803_no2}) & (c) \eqref{803_no3}
\end{tabular}
\end{center}
\caption{Complete improper affine front with total curvature $-8\pi$, genus $0$, and three embedded ends}\label{tc8pi03}
\end{figure}

We will investigate the \tr{remaining} case $(g,n) = (1,1).$
We may assume that $\Sigma = T_\tau\setminus\{[0]\}$, where $T_\tau$ is \tr{as} in \eqref{torus}, and $\rho([0])=\infty$, and identify $\rho, F,$ and $G$ with the elliptic functions on $\C$.
\tr{First, we will list up the possible Weierstrass data without thinking \trr{of} the period condition \eqref{period} and then consider the period problem for each case.} 
Since $\deg\rho=4$, $\rho$ may have  \tr{poles other than $0$}. We shall divide this case into the following 
{\bf (I)\tg{--}(IV)}.

{\bf (I)}
\
The case $\ord_0\rho=-4$

By \eqref{relation}, $G^\pr \neq 0$ on $\Pi_0\setminus\{0\}$ and  $G^\pr(0)\neq \infty$.
In particular, $G^\pr$ is holomorphic on $\Pi_0$, so $G^\pr$ is a non-zero constant.
Hence, $G$ is not elliptic.

{\bf (II)}
\
The case where there exists $p\neq0$ which is a pole of $\rho$.
$$
(\ord_0\rho, \ord_p\rho)=(-1,-3), \ \tr{\text{or}}\ (-2,-2), \tr{\text{or}}\ (-3,-1).
$$

(II-a)\ The case $(\ord_0\rho, \ord_p\rho)=(-1,-3)$.

In this case, since $\ord_pG^\pr=3$ and  $G^\pr \neq 0$ on $\Pi_0\setminus\{0,p\}$, $\ord_0 G^\pr=-3$, and then $\ord_0G=-2$.
Hence, $G$ is given by $G = c\wp \  (c\neq0)$.
Also, $\ord_0F^\pr=-3$ holds. 
Thus, \tr{the} Weierstrass  data is given by
$$
F=a\wp^\pr+b\wp,
\quad
G=c\wp
\quad
(a>0, c\neq0).
$$

(II-b)\ The case $(\ord_0\rho, \ord_p\rho)=(-2,-2)$.

Since $\ord_pG^\pr=2$, we know $\ord_0G^\pr=-2$ and $\ord_0G=-1$. This is impossible.

(II-c) \ The case $(\ord_0\rho, \ord_p\rho)=(-3,-1)$

Since $\ord_pG^\pr=1$, we find that $\ord_0G^\pr=-1$. This is also impossible.

{\bf (III)}
\
The case where there are $p,q\in\Pi_0\setminus\{0\}\ (p\neq q)$ which are poles of $\rho$.

If $\ord_p\rho=\ord_q\rho=-1,$ \too{and}  $\ord_0\rho=-2$, then $\ord_pG^\pr=\ord_qG^\pr=1$, and $\ord_0 G^\pr=-2$, so
$\ord_0G=-1$. This does not happen.
On the other hand, \tr{if} we assume that
$
\ord_p\rho=\ord_0\rho=-1,\ \ord_q\rho=-2\tr{,}
$
\tr{then} we obtain $\ord_pG^\pr=1, \ord_qG^\pr=2$\tr{,} and $\ord_0G^\pr=-3$.
From this, we find that $\ord_0G=-2$. 
Hence, $G$ can be written as $G=c\wp \ (c\neq 0).$
Thus, for the same reason as (II-a), we have Weierstrass  data
$$
F=a\wp^\pr+b\wp,
\quad
G=c\wp
\quad
(a>0, c\neq0).
$$

{\bf (IV)}
\
The case where there are distinct points $p,q,r\in\Pi_0\setminus\{0\}$ which are \tr{simple} poles of $\rho$.
Since $\ord_0G=-2$, we obtain
$$
F=a\wp^\pr+b\wp,
\quad
G=c\wp
\quad
(a>0, c\neq0).
$$

Also, in each case, since $G^\pr(p), G^\pr(q), G^\pr(r)=0$, the points $p,q,$ and $r$ are half periods of $[1,\tau]$ because of $G^\pr=c\wp^\pr$ and Fact \Ref{wp}. 
Thus, the case (II-a) and {\bf (III)} do not happen. 
Indeed, for the case (II-a), $3p\equiv0\mod [1,\tau]$ (\cite{courant}). So $p$ is not a half period, which is impossible. {\bf (III)} is also impossible for the same reason as (II-a). 
Therefore, we only have to consider the case of {\bf (IV)} and may assume
$p=1/2, q=(1+\tau)/2, r=\tau/2$.

Now, we shall consider the period condition \eqref{period}. Direct computations give that
$$
\int FdG=c\left(\fr{a}{30}\wp^{\pr\pr\pr}(z)+\fr{2}{5}ag_2\zeta(z)-\fr{3}{5}ag_3z+\fr{1}{2}b\wp(z)^2\right)
$$
up to additive constant, where $\zeta(z)$ is the Weierstrass $\zeta$-function which satisfies 
$\zeta^\pr(z)=-\wp(z)$ and $\lim_{z\to0}(\zeta(z)-1/z)=0$, and $g_2, g_3$ are \tr{as} in \eqref{g2g3}.
In addition, consider two curves $\gamma_1(t)\coloneqq 1/4+\tau t$ \trr{and}  $\gamma_2(t)\coloneqq t+\tau/4\ (t\in[0,1])$,
which \trr{generate} the fundamental group of $\tr{T_\tau}$, and a loop $\gamma$ around the origin.
Then, it holds that
$$
\int_{\gamma_1}FdG=\fr{1}{5}ac\left(2g_2\eta_2-3g_3\tau\right),
\int_{\gamma_2}FdG=\fr{1}{5}ac\left(2g_2\eta_1-3g_3\right),
\int_{\gamma}FdG
=0,
$$
where $\eta_1,\eta_2$ are constant\tr{s}, which are determined by the lattice $[1,\tau]$, satisfying $\eta_1=\zeta(z+1)-\zeta(z)$ \tr{and} $\eta_2=\zeta(z+\tau)-\zeta(z)$ for all $z\in\C$.
Hence, 
the period \tr{condition} \eqref{period} \tr{is} equivalent to
\begin{equation}\label{perA}
\left\{
\begin{array}{l}
c\left(2g_2\eta_2-3g_3\tau\right)\in i\R,\\
c\left(2g_2\eta_1-3g_3\right)\in i\R. 
\end{array}
\right.
\end{equation}

\begin{proposition}\label{torus8pi1}
Complete improper affine fronts \tr{of} genus $1$ whose total curvature is $-8\pi$ are constructed by the Weierstrass data
\begin{equation}\label{wd8pitorus}
F=a\wp^\pr+b\wp,
\quad
G=c\wp
\quad
(a>0, c\neq0),
\end{equation}
defined on $\C/[1,\tau] \setminus \{[0]\} $ and satisfying the period condition \eqref{perA}.
\end{proposition}

\begin{remark}\label{doesnotshow}
\trr{Note} that Proposition \ref{torus8pi1} still does not show the existence of the surface because we need to determine the modulus $\tau$ of the torus and choose $c\in\C\setminus\{0\}$ that \trr{satisfies} the period condition \eqref{perA}. 
If $\tau=i$ (i.e., $T_\tau$ is \too{the} square torus and this case corresponds with the Chen--Gackstatter 
\trr{surface for} minimal surface case (\trr{\cite{CG_82}})) or $\tau=e^{(2\pi i)/3}$ (i.e., $T_\tau$ is \too{the} \tr{rhombic torus with $60^\circ$-$120^\circ$ angles}), then one can observe that these cases are impossible.
In fact, when $\tau=i$, it holds that $g_2>0, g_3=0$ and $\eta_1=-i \pi$ (\cite[Section 18]{daenkansu}) and then \eqref{perA} yields $c=0$. 
When $\tau=e^{(2\pi i)/3}$, it holds that $g_2=0, g_3>0$ and $\eta_1=2\pi/\rt{3}$.
Then, \eqref{perA} implies $c=0$. 
\end{remark}

From now on, we will show \tr{the} existence of the surface in the special case where $\tau=e^{i\al}\ (\al\in(0,\pi))$ in Proposition \ref{torus8pi1}.
If the period condition \eqref{perA} holds, then one can see that
$$
\Im\left((\overline{2g_2\eta_1-3g_3})(2g_2\eta_2-3g_3\tau)\right)=0.
$$
Since the \trr{invariants} of the $\wp$-function and the $\zeta$-function, namely $g_2, g_3, \eta_1,$ and $\eta_2$ are continuous functions \tr{with respect to} $\tau$, we put $g_2=g_2(\tau), g_3=g_3(\tau), \eta_1=\eta_1(\tau)$, and $ \eta_2=\eta_2(\tau)$. We then set 
\begin{equation}\label{periodfunc1}
P(\al)\coloneqq \Im \left(\overline{p_1(\al)}p_2(\al)\right),
\end{equation}
where $p_1(\al)\coloneqq 2g_2(e^{i\al})\eta_1(e^{i\al})-3g_3(e^{i\al}),\ 
p_2(\al)\coloneqq 2g_2(e^{i\al})\eta_2(e^{i\al})-3g_3(e^{i\al})e^{i\al}.$

\begin{theorem}\label{torus8pi2}
There exists $\al_0\in(\pi/3,\pi/2)$ such that $P(\al_0)=0$. In particular, there exists a complete improper affine front $\psi:\C/[1,e^{i\al_0}] \setminus\{[0]\}\to \R^3$ of genus $1$ whose total curvature is $-8\pi$ (Figure \ref{-8pig1}).
\end{theorem}

\begin{proof}
\cite[Section 18]{daenkansu} shows that the concrete values of $g_2, g_3, \eta_1$ and $\eta_2$ are
$$
g_2\left(\tr{e^{\fr{\pi}{3}i}}\right)=0,\  g_3\left(\tr{e^{\fr{\pi}{3}i}}\right)>0,\  
\eta_1\left(\tr{e^{\fr{\pi}{3}i}}\right)=\fr{2\pi}{\rt{3}}, \ 
\eta_2\left(\tr{e^{\fr{\pi}{3}i}}\right)= \fr{2\pi }{\rt{3}}e^{-\fr{i \pi }{3}},
$$
$$
g_2(i)>0,\  g_3(i)=0, \ \eta_1(i)=\pi, \ \eta_2(i)= -i\pi.
$$
Then, direct computations give that 
$P(\pi/3)=(9\rt{3}g_3^2)/2>0,\ P(\pi/2)=-4g_2^2\pi^2<0.$
Since the function $P(\al)$ is continuous on $(0, \pi)$, the \trr{intermediate} value theorem yields that there exists $\al_0\in(\pi/3,\pi/2)$ such that $P(\al_0)=0$.

Here, either $p_1(\al)$ or $p_2(\al)$ does not vanish for any $\al\in(0,\pi)$. In fact, if $p_1(\al)=p_2(\al)=0$ for some $\al$, then it holds that 
$$
2g_2(e^{i\al})(\eta_1(e^{i\al})e^{i\al}-\eta_2(e^{i\al}))=0.
$$
By the Legendre's identity $\eta_1(e^{i\al})e^{i\al}-\eta_2(e^{i\al})=2\pi i$, one can observe that $g_2(e^{i\al})=0$ and then the torus is \too{the} \tr{rhombic torus with $60^\circ$-$120^\circ$ angles}. From Remark \ref{doesnotshow}, it is impossible.

Thus, \tr{if} we choose \tr{the} complex number $c$ \tr{as} in \eqref{wd8pitorus} \tr{either}
$$
c=i\overline{p_1(\al_0)}\quad\tr{\text{or}}\quad c=i\overline{p_2(\al_0)}\tr{,} 
$$
\tr{whichever is non-zero,} 
then the period \trr{condition} \eqref{perA} is satisfied. Therefore, we complete the proof.
\end{proof}

\begin{figure}[h]
\begin{center}
\begin{tabular}{c@{\hspace{3.6cm}}c}
\includegraphics[width=40mm]{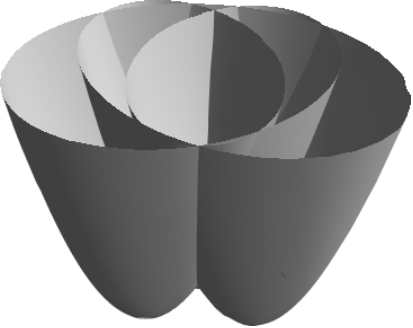}&
\includegraphics[width=40mm]{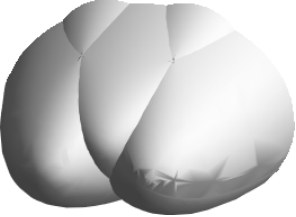}\\
\end{tabular}
\end{center}
\caption{Complete improper affine front of genus $1$ with total curvature $-8\pi$ when $c=i\overline{p_1(\al_0)}$. The values of $\al_0$ and $c$ can be estimated as 
$\al_0\approx 1.37048$, $c=i\overline{p_1(\al_0)}\approx 1265.89 + 370.33i$ by using the Mathematica software.}\label{-8pig1}
\end{figure}

Theorem \ref{torus8pi2} shows that there is a complete improper affine front with the maximum total curvature and positive genus.

\begin{remark}\label{notunique}
Now we consider \tr{the} function 
\begin{equation}
\tilde{P}(\tau)\coloneqq 
\Im\left(\overline{\tilde{p_1}(\tau)}\tilde{p_2}(\tau)\right),
\end{equation}
where 
$\tilde{p_1}(\tau)\coloneqq 2g_2(\tau)\eta_1(\tau)-3g_3(\tau)$ and    
$\tilde{p_2}(\tau)\coloneqq 2g_2(\tau)\eta_2(\tau)-3g_3(\tau)\tau$ 
are  \tg{smooth functions of $
\tau$} defined on the upper half plane $\H\coloneqq\{\tau\in\C\  ;\  \Im\tau>0\}$. 
Theorem \ref{torus8pi2} shows that 
$\tilde{P}(e^{i\al_0})=P(\al_0)=0$. 
On the other hand, the invariants $g_2(\tau), g_3(\tau), \eta_1(\tau),$ and 
$\eta_2(\tau)$ have an expression by the Weierstrass $\theta$-function. \tr{Mathematica software} computes 
$$
\left.\fr{d \tilde{P}}{d\al}(e^{i \al})\right|_{\al=\al_0}\approx -7.74116\times10^6 \neq0.
$$ 
Thus, from the implicit function theorem, there exists an interval $I\ (\ni 0)$ and a smooth curve 
$\phi\col I \to \C$ such that 
$$
\phi(0)=e^{i\al_0}, \quad \tilde{P}(\phi(t))=0\ (t\in I).
$$ 
Hence, when we set $W\coloneqq\left\{\tau=\phi(t)\in\H ; \tilde{P}(\phi(t))=0\ (t\in I)\right\}$ and choose \tr{$c$} \tr{as either}
$i \overline{\tilde{p_1}(\tau)}$ \tr{or} $i\overline{\tilde{p_2}(\tau)}$, \tr{whichever is non-zero}, for each $\tau\in W$, the period condition \eqref{perA} holds. Therefore, it implies the existence of \tr{a} one parameter family with respect to the modulus $\tau$ of complete improper affine fronts of genus $1$ and the total curvature $-8\pi$.

\tg{{\it The complete classification is an open problem in the genus $1$ case. }}
\end{remark}

Finally, we give a new example of a complete improper affine front of genus $1$ (for known examples of genus 1 whose total curvature is $-12\pi$, see \cite[Section 4, No.6]{IA-map}).
\begin{example}
Let $\Sigma=\C/(\Z\oplus i\Z)\setminus\{[0]\}$ be the square torus minus one point and define $F,G$ by
\begin{equation}
F=\wp^{\pr\pr}+\fr{5g_2}{7\pi}\wp,
\quad
G=\wp^{\pr}.
\end{equation}  
One can verify that these $F,G$ satisfy the period condition \eqref{period}. Therefore, $(F,G)$ \tr{gives} a complete improper affine front  $\psi\col \Sigma\to\R^3$ of genus $1$ with the total curvature $-10\pi$ (Figure \ref{-10pi}). 
\end{example}

\begin{figure}[h]
\begin{center}
\begin{tabular}{c@{\hspace{4cm}}c}
\includegraphics[width=37mm]{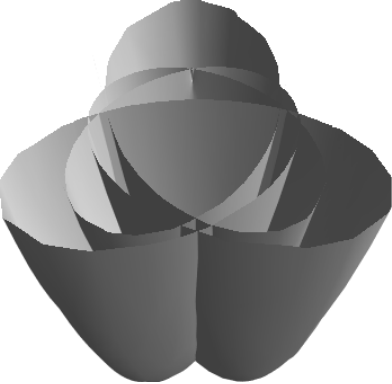}&
\includegraphics[width=32mm]{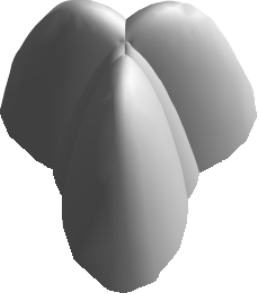}\\
\end{tabular}
\end{center}
\caption{Complete improper affine front of genus $1$ with total curvature $-10\pi$}\label{-10pi}
\end{figure}

\begin{figure}
\begin{center}
\begin{tabular}{c@{\hspace{0.5cm}}c@{\hspace{0.5cm}}c@{\hspace{0.5cm}}c}
\includegraphics[width=26mm]{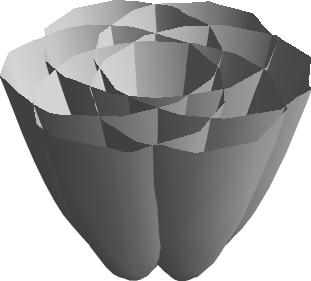}&
\includegraphics[width=26mm]{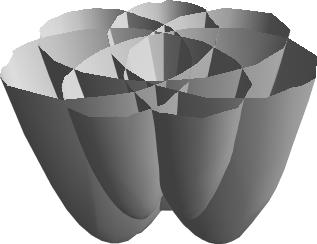}&
\includegraphics[width=26mm]{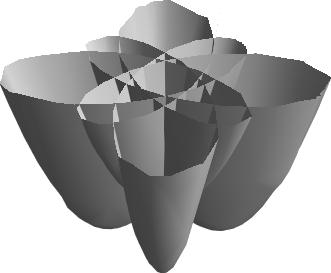}&
\includegraphics[width=26mm]{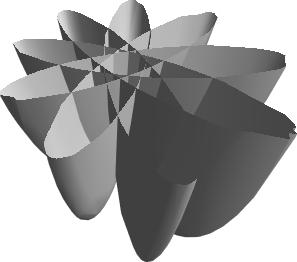}\\
\eqref{801_no1} & \eqref{801_no2} & \eqref{801_no3} & \eqref{801_no4} 
\end{tabular}
\end{center}
\begin{center}
\begin{tabular}{c@{\hspace{1.5cm}}c@{\hspace{1.5cm}}c}
\includegraphics[width=30mm]{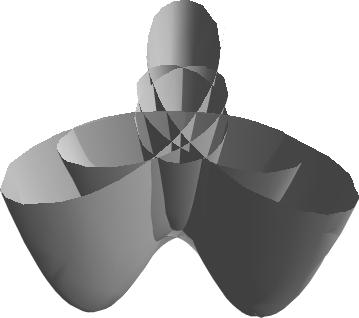}&
\includegraphics[width=30mm]{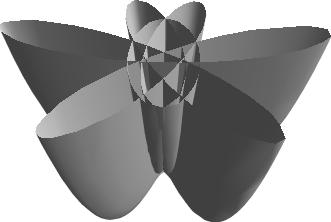}&
\includegraphics[width=20mm]{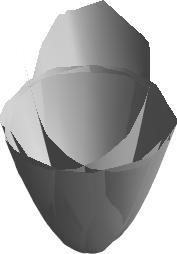}\\
\eqref{801_no5} & \eqref{801_no6} & \eqref{801_no7}
\end{tabular}
\tg{\caption{Complete improper affine front with total curvature $-8\pi$, genus $0$, and one end}}\label{tc8pi01}
\end{center}
\end{figure}

\begin{figure}[h]
\begin{center}
\begin{tabular}{c@{\hspace{0.5cm}}c@{\hspace{0.5cm}}c@{\hspace{0.5cm}}c@{\hspace{0.5cm}}c}
\includegraphics[width=15mm]{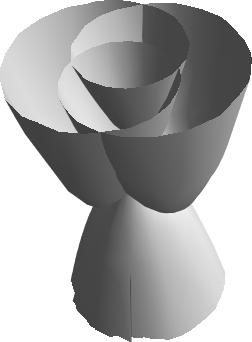}&
\includegraphics[width=12mm]{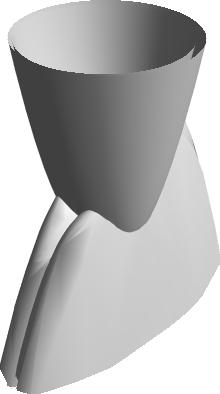}&
\includegraphics[width=20mm]{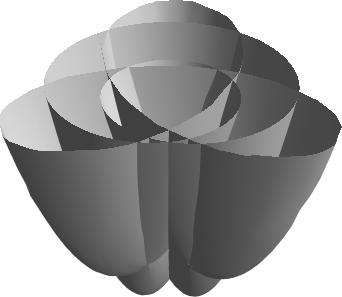}&
\includegraphics[height=18mm]{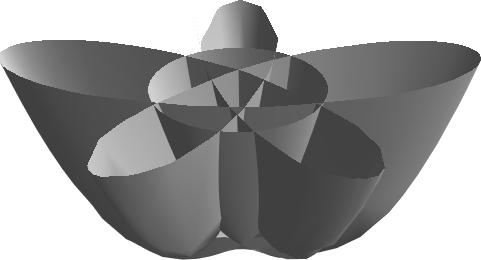}&
\includegraphics[width=15mm]{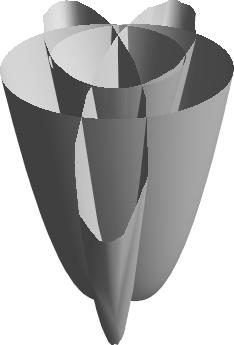}\\
\eqref{802_no1} & \eqref{802_no2} & \eqref{802_no3} & \eqref{802_no4} &\eqref{802_no5}  \end{tabular}
\begin{tabular}{c@{\hspace{0.5cm}}c@{\hspace{0.5cm}}c@{\hspace{0.5cm}}c@{\hspace{0.5cm}}c}
\includegraphics[width=15mm]{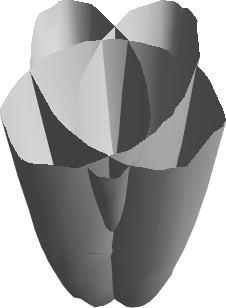}&
\includegraphics[width=18mm]{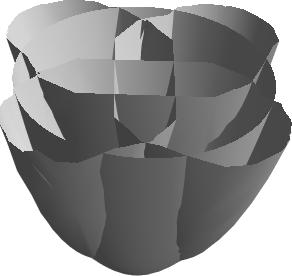}&
\includegraphics[width=25mm]{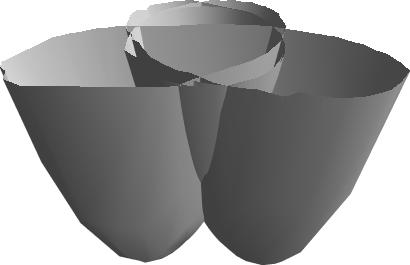}&
\includegraphics[width=25mm]{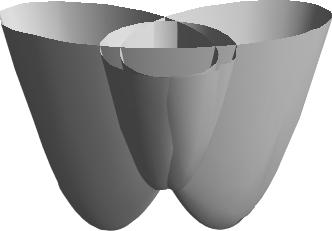}&
\includegraphics[width=20mm]{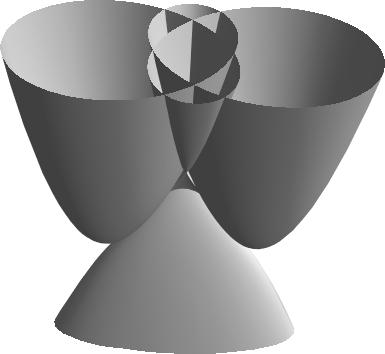}\\
\eqref{802_no6} & \eqref{802_no7} & \eqref{802_no8} & \eqref{802_no9} & \eqref{802_no10}    
\end{tabular}
\begin{tabular}{c@{\hspace{0.5cm}}c@{\hspace{0.5cm}}c@{\hspace{0.5cm}}c@{\hspace{0.5cm}}c}
\includegraphics[width=18mm]{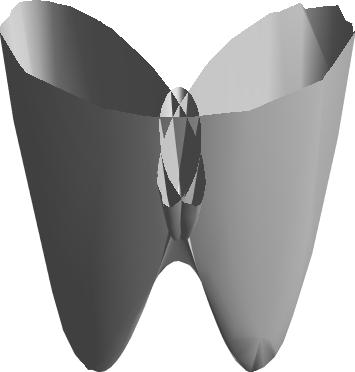}&
\includegraphics[width=26mm]{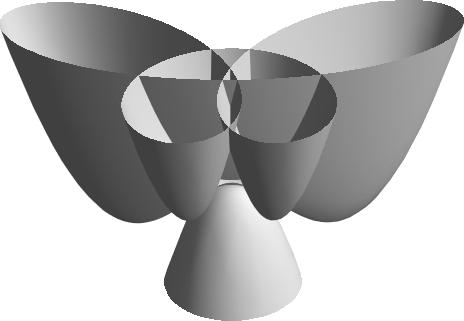}&
\includegraphics[width=20mm]{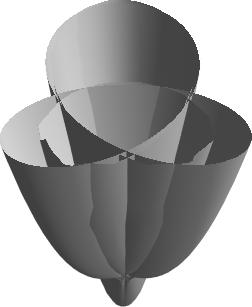}&
\includegraphics[width=20mm]{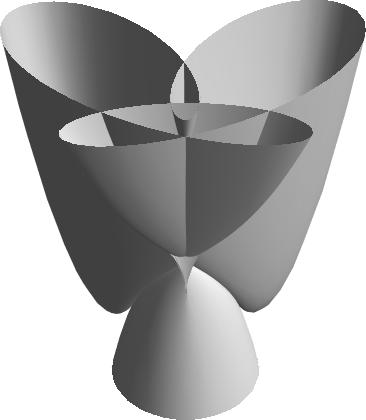}&
\includegraphics[width=16mm]{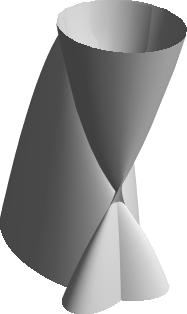}\\
\eqref{802_no11} & \eqref{802_no12} & \eqref{802_no13} & \eqref{802_no14} & \eqref{802_no15}
\end{tabular}
\tg{\caption{Complete improper affine front with total curvature $-8\pi$, genus $0$, and two ends}}\label{tc8pi02}
\end{center}
\end{figure}


\newcommand{\etalchar}[1]{$^{#1}$}

\end{document}